\documentclass[a4paper,11pt]{amsart}

\usepackage{amsmath,amsfonts,amssymb,amsthm,amscd}
\usepackage{graphicx}
\usepackage{perpage}
\usepackage{url}
\usepackage{color}
\usepackage[T1]{fontenc}
\usepackage{hyperref}
\usepackage[a4paper,scale={0.72,0.74},marginratio={1:1},footskip=7mm,headsep=10mm]{geometry}
\usepackage[latin9]{inputenc}
\usepackage{mathrsfs}
\usepackage{tikz}
\numberwithin{equation}{section}

\newtheorem{theorem}{Theorem}[section]
\newtheorem{lemma}[theorem]{Lemma}

\newtheorem{corollary}[theorem]{Corollary}
\newtheorem{remark}[theorem]{Remark}
\newtheorem{definition}[theorem]{Definition}

\newtheorem{conjecture}[theorem]{Conjecture}

\numberwithin{equation}{section}

\newcommand{\R}{\mathbb{R}}

\newcommand{\Z}{\mathbb{Z}}
\newcommand{\N}{\mathbb{N}}

\newcommand{\PP}{\mathbb{P}}
\newcommand{\EE}{\mathbb{E}}

\newcommand{\dmin}{d_\mathrm{min}}
\newcommand{\dave}{d_\mathrm{ave}}

\newcommand{\cH}{\mathcal{H}}
\newcommand{\cP}{\mathcal{P}}
\newcommand{\cC}{\mathcal{C}}
\newcommand{\cL}{\mathcal{L}}

\newcommand{\cT}{\mathcal{T}}

\newcommand{\stab}{\mathrm{stab}}
\newcommand{\meta}{\mathrm{meta}}
\newcommand{\CM}{\mathrm{CM}}

\begin{document}

\title{Metastability for Glauber dynamics on random graphs}

\author{S.\ Dommers}
\address{Department of Mathematics, Ruhr University Bochum,
Universit\"atsstra\ss e 150, 44780 Bochum, Germany}
\email{sander.dommers@ruhr-uni-bochum.de}

\author{F.\ den Hollander}
\address{Mathematical Institute, Leiden University, P.O.\ Box 9512,
2300 RA Leiden, The Netherlands.}
\email{denholla@math.leidenuniv.nl}

\author{O.\ Jovanovski}
\address{Mathematical Institute, Leiden University, P.O.\ Box 9512,
2300 RA Leiden, The Netherlands.}
\email{o.jovanovski@math.leidenuniv.nl}

\author{F.R.\ Nardi}
\address{Eindhoven University of Technology, Department of Mathematics
and Computer Science, P.O.\ Box 513, 5600 MB Eindhoven,
The Netherlands.}
\email{f.r.nardi@tue.nl}

\begin{abstract}
In this paper we study metastable behaviour at low temperature of Glauber spin-flip
dynamics on random graphs. We fix a large number of vertices and randomly allocate
edges according to the Configuration Model with a prescribed degree distribution.
Each vertex carries a spin that can point either up or down. Each spin interacts
with a positive magnetic field, while spins at vertices that are connected by edges
also interact with each other via a ferromagnetic pair potential. We start from the
configuration where all spins point down, and allow spins to flip up or down
according to a Metropolis dynamics at positive temperature. We are interested in
the time it takes the system to reach the configuration where all spins point up. In
order to achieve this transition, the system needs to create a sufficiently large
droplet of up-spins, called critical droplet, which triggers the crossover.

In the limit as the temperature tends to zero, and subject to a certain \emph{key
hypothesis} implying metastable behaviour, the average crossover time
follows the classical \emph{Arrhenius law}, with an exponent and a prefactor
that are controlled by the \emph{energy} and the \emph{entropy} of the critical
droplet. The crossover time divided by its average is exponentially distributed.
We study the scaling behaviour of the exponent as the number of vertices tends
to infinity, deriving upper and lower bounds. We also identify a regime for the
magnetic field and the pair potential in which the key hypothesis is satisfied.
The critical droplets, representing the saddle points for the crossover, have a
size that is of the order of the number of vertices. This is because the random
graphs generated by the Configuration Model are expander graphs.
\end{abstract}

\thanks{The research in this paper was supported by NWO Gravitation Grant
024.002.003-NETWORKS. The work of SD was supported also by DFG Research
Training Group 2131.}

\date{\today}

\maketitle


\section{Introduction and main theorems}

A physical system is in a \emph{metastable state} when it remains locked for a very long
time in a phase that is different from the one corresponding to thermodynamic equilibrium.
The latter is referred to as the \emph{stable state}. Classical examples are supersaturated
vapours, supercooled liquids, and ferromagnets in the hysteresis loop. The main three objects
of interest for metastability are the transition time from the metastable state to the stable state,
the gate of configurations the system has to cross in order to achieve the transition, and the
tube of typical trajectories the system follows prior to and after the transition.

Metastability for interacting particle systems on \emph{lattices} has been studied intensively
in the past three decades. Various different approaches have been proposed. After initial
work by Cassandro, Galves, Olivieri and Vares~\cite{CGOV84}, Neves and
Schonmann~\cite{NS91}, \cite{NS92}, a powerful method -- known as the \emph{pathwise
approach} to metastability based on large deviation theory -- was developed in Olivieri and
Scoppola~\cite{OS95}, \cite{OS96}, Catoni and Cerf~\cite{CC97}, Manzo, Nardi, Olivieri
and Scoppola~\cite{MNOS04}, Cirillo and Nardi~\cite{CN13}, Cirillo, Nardi and Sohier~\cite{CNSo15}. 
This was successfully applied to low-temperature Ising and Blume-Capel models subject to 
Glauber spin-flip dynamics (in two and three dimensions, with isotropic, anisotropic and 
staggered interactions) in Koteck\'y and Olivieri~\cite{KO92}, \cite{KO93}, \cite{KO94}, Cirillo 
and Olivieri~\cite{CO96}, Ben Arous and Cerf~\cite{BAC96}, Nardi and Olivieri~\cite{NO96}. 
Later, another powerful method -- known as the \emph{potential-theoretic approach} to 
metastability based on the analogy between Markov processes and electric networks -- was 
developed in Bovier, Eckhoff, Gayrard and Klein~\cite{BEGK00}, \cite{BEGK01}, \cite{BEGK02}, 
\cite{BEGK04}. This was shown in Bovier and Manzo~\cite{BM02}, Bovier, den Hollander and 
Spitoni~\cite{BdHS10} to lead to a considerable sharpening of earlier results. For other approaches 
to metastability, as well as further examples of metastable stochastic dynamics and relevant 
literature, we refer the reader to the monographs by Olivieri and Vares~\cite{OV05}, Bovier 
and den Hollander~\cite{BdH15}.

Recently, there has been interest in the Ising model on \emph{random graphs} (Dembo and
Montanari~\cite{DM10}, Dommers, Giardin\`a and van der Hofstad~\cite{DGvdH10}, Mossel
and Sly~\cite{MS13}). The only results known to date about metastability subject to Glauber
spin-flip dynamics are valid for $r$-regular random graphs (Dommers~\cite{Dpr}). In the present
paper we investigate what can be said for more general degree distributions. Metastability
is much more challenging on random graphs than on lattices. Moreover, we need to capture
the metastable behaviour for a \emph{generic realisation} of the random graph.

In Section~\ref{IsingGlauber} we define the Ising model on a random multigraph subject to
Glauber spin-flip dynamics. We start from the configuration where all spins point down, and
allow spins to flip up or down according to a Metropolis dynamics at positive temperature.
We are interested in the time it takes the system to reach the configuration where all spins
point up. In Section~\ref{metastability} we introduce certain geometric quantities that play a
central role in the description of the metastable behaviour of the system, and state three
general theorems that are valid under a certain key hypothesis. These theorems concern the
average transition time, the distribution of the transition time, and the gate of saddle point
configurations for the crossover, all in the limit of low temperature. They involve certain key
quantities associated with the random graph. Our goal is to study the scaling behaviour of
these quantities as the size of the graph tends to infinity.

In Section~\ref{literature} we describe four examples to which the three general theorems
apply: three refer to regular lattices, while one refers to the Erd\H{o}s-R\'enyi random graph.
In Section~\ref{configurationmodel} we recall the definition of the Configuration Model, which
is an example of a random graph with a non-trivial geometric structure. In Section~\ref{maintheorems}
we state our main metastability results for the latter. In Section~\ref{discussion} we place these
results in their proper context and give an outline of the remainder of the paper.


\subsection{Ising model and Glauber dynamics}
\label{IsingGlauber}

Given a finite connected non-oriented multigraph $G=(V,E)$, let $\Omega=\{-1,+1\}^V$
be the set of configurations $\xi=\{\xi(v)\colon\,v \in V\}$ that assign to each vertex
$v\in V$ a spin-value $\xi(v)\in\{-1,+1\}$. Two configurations that will be of particular
interest to us are those where all spins point up, respectively, down:
\begin{equation}
\boxplus \equiv +1, \qquad \boxminus \equiv -1.
\end{equation}
For $\beta \geq 0$, playing the role of \emph{inverse temperature}, we define the Gibbs
measure
\begin{equation}
\mu_\beta(\xi) = \frac{1}{Z_\beta}\,e^{-\beta \cH(\xi)},
\qquad \xi\in\Omega,
\label{eq:Gibbs}
\end{equation}
where $\cH\colon\,\Omega\to\R$ is the \emph{Hamiltonian} that assigns an energy
to each configuration given by
\begin{equation}
\cH(\xi) = -\frac{J}{2} \sum_{(v,w) \in E} \xi(v)\xi(w)
-\frac{h}{2} \sum_{v \in V} \xi(v), \qquad \xi\in\Omega,
\label{eq:hamiltonian}
\end{equation}
with $J>0$ the \emph{ferromagnetic pair potential} and $h>0$ the \emph{magnetic field}.
The first sum in the right-hand side of \eqref{eq:hamiltonian} runs over all non-oriented
edges in $E$. Hence, if $v,w \in V$ have $k \in \N_0$ edges between them, then their joint
contribution to the energy is $-k\,\frac{J}{2}\,\xi(v)\xi(w)$.

We write $\xi\sim\zeta$ if and only if $\xi$ and $\zeta$ agree at all but one vertex. A
transition from $\xi$ to $\zeta$ corresponds to a flip of a single spin, and is referred
to as an \emph{allowed move}. Glauber spin-flip dynamics on $\Omega$ is the
continuous-time Markov process $(\xi_t)_{t \geq 0}$ defined by the transition rates
\begin{equation}
c_\beta(\xi,\zeta) = \begin{cases}
e^{-\beta[\cH(\zeta)-\cH(\xi)]_+}, &\xi\sim\zeta,\\
0, &\mbox{otherwise}.
\end{cases}
\end{equation}
The Gibbs measure in \eqref{eq:Gibbs} is the reversible equilibrium of this dynamics.
We write $P^{G,\beta}_\xi$ to denote the law of $(\xi_t)_{t \geq 0}$ given $\xi_0=\xi$,
$\cL^{G,\beta}$ to denote the associated generator, and $\lambda^{G,\beta}$ to
denote the principal eigenvalue of $\cL^{G,\beta}$. The upper indices $G,\beta$
exhibit the dependence on the underlying graph $G$ and the interaction strength
$\beta$ between neighbouring spins. For $A\subseteq\Omega$, we write
\begin{equation}
\tau_A = \inf\big\{t>0\colon\,\xi_t \in A,\,\exists\,0<s<t\colon\,\xi_s \neq \xi_0\big\}
\end{equation}
to denote the first hitting time of the set $A$ after the starting configuration is left.


\subsection{Metastability}
\label{metastability}

To describe the metastable behaviour of our dynamics we need the following
geometric definitions.

\begin{definition}
(a) The communication height between two distinct configurations $\xi,\zeta\in\Omega$ is
\begin{equation}
\Phi(\xi,\zeta) = \min_{\gamma\colon\,\xi\to\zeta} \max_{\sigma\in\gamma}
\cH(\sigma),
\end{equation}
where the minimum is taken over all paths $\gamma\colon\,\xi\to\zeta$ consisting of
allowed moves only. The communication height between two non-empty disjoint sets
$A,B\subset\Omega$ is
\begin{equation}
\Phi(A,B) = \min_{\xi\in A,\zeta\in B} \Phi(\xi,\zeta).
\end{equation}
(b) The stability level of $\xi\in\Omega$ is
\begin{equation}
V_\xi = \min_{ {\zeta\in\Omega:} \atop {\cH(\zeta)<\cH(\xi)} } \Phi(\xi,\zeta)-\cH(\xi).
\end{equation}
(c) The set of stable configurations is
\begin{equation}
\Omega_\stab = \left\{\xi\in\Omega\colon\,\cH(\xi)
= \min_{\zeta\in\Omega} \cH(\zeta)\right\}.
\end{equation}
(d) The set of metastable configurations is
\begin{equation}
\Omega_\meta = \left\{\xi\in\Omega\backslash\Omega_\stab\colon\,
V_\xi = \max_{\zeta\in\Omega\backslash\Omega_\stab} V_\zeta\right\}.
\end{equation}
\end{definition}

It is easy to check that $\Omega_\stab = \{\boxplus\}$ for all $G$ because $J,h>0$.
For general $G$, however, $\Omega_\meta$ is not a singleton, but we will be interested
in those $G$ for which the following \emph{hypothesis} is satisfied:

\medskip

\begin{itemize}
\item [(H)] $\Omega_\meta = \{\boxminus\}$.
\label{itm:Hhyp}
\end{itemize}

\medskip\noindent
The energy barrier between $\boxminus$ and $\boxplus$ is
\begin{equation}
\Gamma^\star = \Phi(\boxminus,\boxplus)-\cH(\boxminus).
\end{equation}

\begin{definition}
Let $(\cP^\star,\cC^\star)$ be the unique maximal subset of $\Omega\times\Omega$
with the following properties (see Fig.~{\rm \ref{fig-protocrcr}}):
\begin{enumerate}
\item
$\forall\,\xi\in\cP^\star\,\exists\,\xi\prime\in\cC^\star\colon\,\xi\sim\xi\prime$,\\
$\forall\,\xi\prime\in\cC^\star\,\exists\,\xi\in\cP^\star\colon\,\xi\prime\sim\xi$.
\item
$\forall\,\xi\in\cP^\star\colon\,\Phi(\xi,\boxminus)<\Phi(\xi,\boxplus)$.
\item
$\forall\xi\,\in\cC^\star\,\exists\,\gamma\colon\,\xi\to\boxplus\colon\,\\
{\rm (i)} \max_{\zeta\in\gamma} \cH(\zeta)-\cH(\boxminus) \leq \Gamma^\star$.\\
{\rm (ii)} $\gamma \cap\{\zeta\in\Omega\colon\,\Phi(\zeta,\boxminus)
<\Phi(\zeta,\boxplus)\} = \emptyset$.
\end{enumerate}
\end{definition}

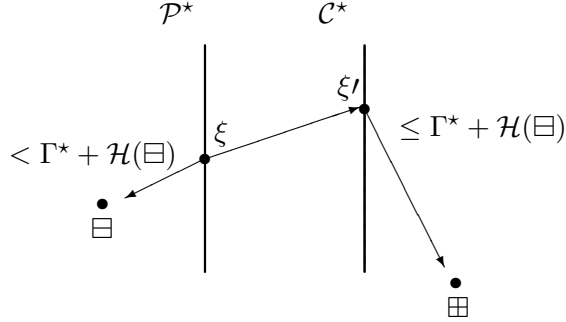
\begin{figure}[htbp]
\vspace{-0.5cm}
\begin{center}
\setlength{\unitlength}{0.3cm}
\begin{picture}(15,15)(0,-1)
{\thicklines
\qbezier(0,0)(0,5)(0,10)
\qbezier(7,0)(7,5)(7,10)
}
\put(-2,11){$\cP^\star$}
\put(5,11){$\cC^\star$}
\put(0.4,5.8){$\xi$}
\put(5.8,7.8){$\xi\prime$}
\put(-5,1.5){$\boxminus$}
\put(10.6,-2){$\boxplus$}
\put(-8.6,4.8){$<\Gamma^\star+\cH(\boxminus)$}
\put(8.5,6){$\leq \Gamma^\star+\cH(\boxminus)$}
\put(0,5){\circle*{.45}}
\put(7,7.2){\circle*{.45}}
\put(-4.5,3){\circle*{.45}}
\put(11,-.5){\circle*{.45}}
\put(0,5){\vector(3,1){6.8}}
\put(7,7.2){\vector(1,-2){3.5}}
\put(0,5){\vector(-2,-1){3.5}}
\end{picture}
\end{center}
\caption{\small Schematic picture of the protocritical set $\cP^\star$ and the critical
set $\cC^\star$.}
\label{fig-protocrcr}
\end{figure}

\noindent
Think of $\cP^\star$ as the set of configurations where the dynamics, on its
way from $\boxminus$ to $\boxplus$, is `almost at the top', and of $\cC^\star$
as the set of configurations where it is `at the top and capable over crossing
over'. We refer to $\cP^{\star}$ as the \emph{protocritical set} and to $\cC^\star$
as the \emph{critical set}. Uniqueness follows from the observation that if
$(\cP_1^\star,\cC_1^\star)$ and $(\cP_2^\star,\cC_2^\star)$ both satisfy
conditions (1)--(3), then so does $(\cP_1^\star\cup\cP_2^\star,\cC_1^\star
\cup\cC_2^\star)$. Note that
\begin{equation}
\begin{array}{lll}
&\cH(\xi)<\Gamma^\star+\cH(\boxminus) &\forall\,\xi \in \cP^\star,\\[0.1cm]
&\cH(\xi)=\Gamma^\star+\cH(\boxminus) &\forall\,\xi \in \cC^\star.
\end{array}
\end{equation}

It is shown in Bovier and den Hollander \cite[Chapter 16]{BdH15} that \emph{subject
to hypothesis} (H) the following three theorems hold.

\begin{theorem}
\label{thm:critdrop}
$\lim_{\beta\to\infty} P^{G,\beta}_\boxminus(\tau_{\cC^\star}<\tau_\boxplus
\mid \tau_\boxplus < \tau_\boxminus)=1$.
\end{theorem}

\begin{theorem}
\label{thm:nucltime}
There exists a $K^\star \in (0,\infty)$ such that
\begin{equation}
\lim_{\beta\to\infty}
e^{-\beta\Gamma^\star}\,E^{G,\beta}_\boxminus(\tau_\boxplus) = K^\star.
\end{equation}
\end{theorem}

\begin{theorem}
\label{thm:explaw}
(a) $\lim_{\beta\to\infty} \lambda^{G,\beta}_\beta\,E^{G,\beta}_\boxminus(\tau_\boxplus)=1$.\\
(b) $\lim_{\beta\to\infty} P^{G,\beta}_\boxminus(\tau_\boxplus/
E^{G,\beta}_\boxminus(\tau_\boxplus)>t) = e^{-t}$ for all $t \geq 0$.
\end{theorem}

\noindent
The proofs of Theorems~\ref{thm:critdrop}--\ref{thm:explaw} in \cite{BdH15} do not rely on the
details of the graph $G$, provided it is finite, connected and non-oriented (i.e., allowed
moves are possible in both directions). For concrete choices of $G$, the task is to verify
hypothesis (H) and to identify the triple (see Fig.~\ref{fig-abstractdoublewell})
\begin{equation}
\label{eq:triple}
\big(\cC^\star,\Gamma^\star,K^\star\big).
\end{equation}
For lattice graphs this task has been carried out successfully (even for several classes
of dynamics: see \cite[Chapters 17--18]{BdH15}). For random graphs, however, the triplet
in \eqref{eq:triple} is random, and describing it represents a \emph{very serious challenge}.
In what follows we focus on a particular class of random graphs called the \emph{Configuration
Model}. But before doing so, we first summarise what is known in the literature.

\begin{figure}[htbp]
\vspace{2cm}
\begin{center}
\setlength{\unitlength}{0.4cm}
\begin{picture}(8,6)(0,0)
\put(0,0){\line(11,0){11}}
\put(0,0){\line(0,9){9}}
\qbezier[30](3.3,3.8)(3.3,1.9)(3.3,0)
\qbezier[50](4.8,5.6)(4.8,3)(4.8,0)
\qbezier[25](6.8,3.0)(6.8,1.5)(6.8,0)
\qbezier[40](0,5.8)(2.5,5.8)(4.8,5.8)
{\thicklines
\qbezier(2,8)(3,2)(4,5)
\qbezier(4,5)(5,7)(6,4)
\qbezier(6,4)(7,2)(8,5)
}
\put(2.2,-1.3){$\Omega_\mathrm{meta}$}
\put(4.55,-.9){$\cC^\star$}
\put(5.8,-1.3){$\Omega_\mathrm{stab}$}
\put(11.5,-.3){$\xi$}
\put(-1.3,9.5){$\cH(\xi)$}
\put(-1.4,5.6){$\Gamma^\star$}
\put(3.3,4){\circle*{.35}}
\put(4.8,5.75){\circle*{.35}}
\put(6.8,3.20){\circle*{.35}}
\end{picture}
\end{center}
\vspace{0.3cm}
\caption{\small Schematic picture of $\cH$, $\Omega_\mathrm{meta}$,
$\Omega_\mathrm{stab}$ and $\cC^\star$.}
\label{fig-abstractdoublewell}
\end{figure}
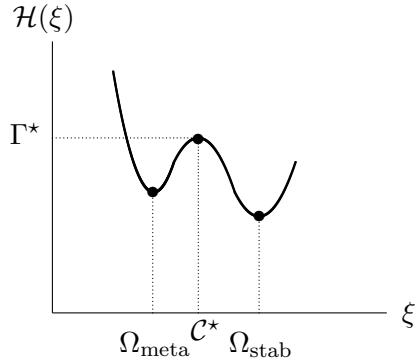


\subsection{Examples of applications}
\label{literature}


\subsubsection{Torus}
If the underlying graph is a torus, then the computations needed to identify the
critical set $\cC^\star$ and the prefactor $K^\star$ simplify considerably. As shown
in Bovier and den Hollander~\cite[Chapter 17]{BdH15}, for Glauber dynamics on
a finite box $\Lambda\subset\Z^2$ (wrapped around to form a torus), the set
$\cC^\star$ consists of all $\ell_c \times (\ell_c -1)$ quasi-squares (located
anywhere in $\Lambda$ in any of the two orientations) with an extra vertex
attached to one of its longest sides, where $\ell_c = \lceil \tfrac{2J}{h}\rceil$ (the
upper integer part of $\tfrac{2J}{h}$). Hypothesis (H) has been verified, and the
exponent and the prefactor equal
\begin{equation}
\Gamma^\star = J(4\ell_c)-h(\ell_c(\ell_c-1)+1), \qquad
K^\star = \frac{1}{|\Lambda|}\,\frac{1}{\frac{4}{3}(2\ell_c-1)}.
\end{equation}
Metastable behaviour occurs if and only if $\ell_c \in (1,\infty)$, and for reasons of
parity it is assumed that $\tfrac{2J}{h} \notin \N$. Similar results apply for a torus in
$\Z^3$.


\subsubsection{Hypercube}
For Glauber dynamics on the $n$-dimensional hypercube, Jovanovski~\cite{Jpr} gives
a complete description of the set $\cC^\star$ (under the assumption that $\tfrac{h}{J}
\neq\frac{a}{b}$ for some $a\in \mathbb{N}$ and $b\in \left\lbrace 1,2,\ldots,2^{n}\right\rbrace$)
and shows that
\begin{equation}
\Gamma_n^\star = \tfrac{1}{3}\left(1-\tfrac{h}{J}+\left\lceil \tfrac{h}{J}\right\rceil \right)
\left(2^{\left\lceil n-\tfrac{h}{J}\right\rceil }-4+2\epsilon\right)-\epsilon, \qquad
K_n^\star = \frac{\left\lceil \tfrac{h}{J}\right\rceil !}{n!\,2^{n-4}\left(3-\epsilon\right)},
\end{equation}
with $\epsilon=\left\lceil n-h\right\rceil \mbox{mod }2$. Hypothesis (H) has been verified.


\subsubsection{Complete graph}
For Glauber dynamics on the complete graph $K_n$, it is easy to see that any
monotone path from $\boxminus$ to $\boxplus$ is an optimal path. It is straightforward
to show that $\cC^\star=\{U\subseteq V\colon\,|U|=n^\star\}$ with $n^\star = \lceil\frac{1}{2}
(n-1-\tfrac{h}{J})\rceil$, whenever $\frac{h}{J}$ is not an integer, and to compute
\begin{equation}
\label{GKexp}
\Gamma_n^\star=n^\star(J(n-n^\star)-h), \qquad  K_n^\star=\frac{1}{|\cC^\star|}
\frac{n}{n-n^\star}.
\end{equation}
Metastable behaviour occurs for any value of $h$ and $J$, provided $n$ is large
enough. Hypothesis (H) is also easy to confirm by observing that every configuration
lies on some optimal path. Like the hypercube, $K_n$ is an expander graph and
consequently the communication height $\Gamma^\star$ grows at least linearly with
the number of vertices (quadratically for $K_n$).

We can reduce the quadratic growth by introducing an interaction parameter that is
inversely proportional to the size of the graph: e.g.\ $J=\tfrac{J'}{n}$ for some
constant $J'>0$, with $h>0$ fixed. It follows that
\begin{equation}
\Gamma_{n}^\star=n^\star\left(J'\left(\frac{n-n^\star}{n}-h\right)\right),
\end{equation}
where this time $n^\star=\lceil\frac{n}{2}(1-\tfrac{h}{J'})-\frac{1}{2}\rceil$, and $K_n^\star$
is the same as in \eqref{GKexp}. Metastable behaviour occurs if and only if $\tfrac{h}{J'}
<1-\frac{1}{n}$.


\subsubsection{Erd\H{o}s-R\'enyi random graph}
Sharp results of the above type become infeasible when the graph is random. The
Erd\H{o}s-R\'enyi random graph is the result of performing bond percolation on
the complete graph, and is a toy model of a graph with a random geometry. Let
$\mathrm{ER}_n(p)$ denotes the resulting random graph on $n$ vertices with
percolation parameter $p=f(n)/n$ for some $f(n)$ satisfying $\lim_{n\to\infty} f(n)
=\infty$, the so-called \emph{dense} Erd\H{o}s-R\'enyi random graph. Then,
as shown in the appendix, metastable behaviour occurs for any $h,J>0$, and
\begin{equation}
\lim_{n\to\infty} \frac{\Gamma_n^\star}{\tfrac14 Jnf(n)} = 1 \quad
\text{ in distribution under the law of }  \mathrm{ER}_n(f(n)/n),
\label{eq:er1}
\end{equation}
which is accurate up to leading order. The computation of $\cC_n^\star$ and $K_n^\star$,
however, is a \emph{formidable task}. The reason for this is that, while \eqref{eq:er1}
allows for a small error in the energy, the set $\cC_n^\star$ is made up of configurations
that have \emph{exactly} the critical energy $\Gamma_n^\star$.

When $f(n)=\lambda$ for some constant $\lambda>1$, the \emph{sparse case}, an
analysis similar to the one carried out in this paper can be used to obtain lower and
upper bounds on the communication height. However, we have been unable to prove
a convergence of the form in \eqref{eq:er1}.


\subsection{Configuration Model}
\label{configurationmodel}

In this section we recall the construction of the random \emph{multi-graph} known as the
\emph{Configuration Model} (illustrated in Fig.~\ref{fig:CM}). We refer to van der
Hofstad~\cite[Chapter 7]{vdHpr} for further details.

\begin{figure}[htbp]
\begin{center}
\includegraphics[scale = 0.10]{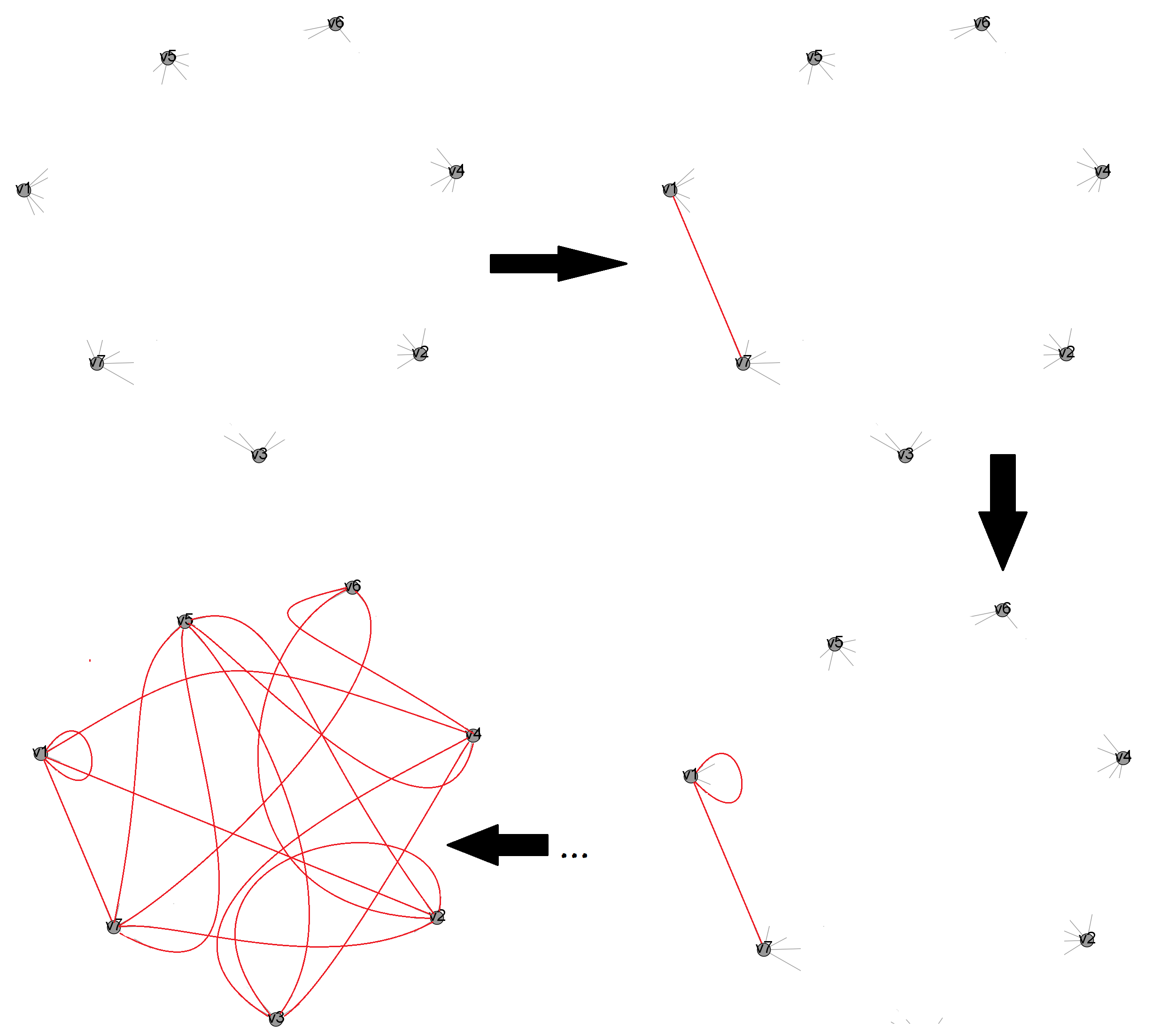}
\end{center}
\caption{\small
Illustration of the construction of $\CM_n$. Three steps in the matching of stubs for
$n=7$ and degree sequence $(5,5,4,5,5,3,5)$.}
\label{fig:CM}
\end{figure}

Fix $n\in\N$, and let $V=\{v_1,\ldots,v_n\}$. With each vertex $v_i$ we associate a
\emph{random degree} $D_i$, in such a way that $D_1,\ldots,D_n \in \N$ are i.i.d.\
with marginal probability distribution $f$ conditional on the event $\{\sum_{i=1}^n D_i
= \mbox{even}\}$. Consider a uniform matching of the elements in the set of
\emph{stubs} (also called half-edges), written
\begin{equation}
\{x_{i,j}\}_{1 \leq i \leq n,1 \leq j \leq D_i}.
\end{equation}
By erasing the second label of the stubs, we can associate with it a
\emph{multi-graph} $\CM_n$ satisfying the requirement that the degree of $v_i$ is
$D_i$ for $1 \leq i \leq n$. The total number of edges is $\tfrac12\sum_{i=1}^n D_i$.

Throughout the sequel we use the symbol $\PP_n$ to denote the law of the random
multi-graph $\CM_n$ on $n$ vertices generated by the Configuration Model. To avoid
degeneracies we assume that
\begin{equation}
\dmin = \min\{k\in\N\colon\,f(k)>0\} \geq 3,
\qquad \dave = \sum_{k\in\N} kf(k) < \infty,
\end{equation}
i.e., all degrees are at least three and the average degree is finite. In this case the graph
is connected \emph{with high probability} (w.h.p.), i.e., with a probability tending to $1$ as
$n\to\infty$ (see van der Hofstad~\cite{vdHpr}).

\subsection{Main theorems}
\label{maintheorems}

We are interested in proving hypothesis (H) and identifying the key quantities
in \eqref{eq:triple} for $G = \CM_n$, which we henceforth denote by
$(\cC^\star_n,\Gamma^\star_n,K^\star_n)$, in the limit as $n\to\infty$.

Our first main theorem settles hypothesis (H) for small magnetic field.

\begin{theorem}
\label{thm:hyp}
Suppose that the inequality in equation \eqref{eq:Hcond1} holds. Then
\begin{equation}
\lim_{n\to\infty} \PP_n\big(\CM_n \text{ satisfies {\rm (H)}}\big) = 1.
\end{equation}
\end{theorem}

Our second and third main theorem provide upper and lower bounds on $\Gamma^\star_n$.
Label the vertices of the graph so that their degrees satisfy $d_{1}\leq\ldots\leq d_{n}$.
Let $\gamma\colon\,\boxminus\to\boxplus$ be the path that successively flips the
vertices $v_{1},\ldots,v_{n}$ (in that order), and let $\ell_{m}=\sum_{i=1}^{m}d_{i}$.

\begin{theorem}
\label{thm:Ubound}
Define
\begin{equation}\label{eq-defmbar}
\bar{m} = \min\left\{1 \leq m \leq n\colon\,\ell_m\left(1-\frac{\ell_m}{\ell_n}\right)
\geq \ell_{m+1}\left(1-\frac{\ell_{m+1}}{\ell_n}\right)-\frac{h}{J}\right\} < \frac{n}{2}.
\end{equation}
Then, w.h.p.,
\begin{equation}
\Gamma_n^{\star} \leq \Gamma_n^+,
\qquad \Gamma_n^+ = J\ell_{\bar{m}}\Big(1-\frac{\ell_{\bar{m}}}{\ell_n}\Big)
-h\bar{m} \pm O\big(\ell_n^{3/4}\big).
\end{equation}
\end{theorem}

For $0<x\leq\tfrac12$ and $\delta>1$, define (see Fig.~\ref{fig:functionI})
\begin{equation}
\begin{aligned}
I_{\delta}\left(x\right) &= \inf\Big\{0<y\leq x\colon\\
&1<x{}^{x\left(1-1/\delta\right)}
\left(1-x\right)^{\left(1-x\right)\left(1-1/\delta\right)}
\left(1-x-y\right)^{-\left(1-x-y\right)/2}\left(x-y\right)^{-\left(x-y\right)/2}y^{-y}\Big\}.
\end{aligned}
\label{eq:functionI}
\end{equation}

\begin{figure}[htbp]
\begin{center}
\begin{tikzpicture}[thick,scale=1, every node/.style={scale=1}]
\draw[->] (0,0) -- (5,0) node[anchor=north] {$x$};
\draw[->] (0,0) -- (0,4) node[anchor=east] {$I_{\delta}$};
\draw	(0,0) node[anchor=north] {0}		
		(2,0) node[anchor=north] {$\tfrac{1}{4}$}	
		(4,0) node[anchor=north] {$\tfrac{1}{2}$};
\draw	(0,1.4) node[anchor=east] {$\tfrac{11}{250}$}	
		(0,2.8) node[anchor=east] {$\tfrac{11}{125}$};		
\draw [black] plot [smooth] coordinates {(0,0) (1/5,30*0.0111) (2/5,30*0.0205) (3/5,30*0.0286)
(4/5,30*0.0361) (5/5,30*0.0429) (6/5,30*0.0490) (7/5,30*0.0546) (8/5,30*0.0596) (9/5,30*0.0641)
(10/5,30*0.0683) (11/5,30*0.0721) (12/5,30*0.0753)  (13/5,30*0.0780) (14/5,30*0.0805)
(15/5,30*0.0825) (16/5,30*0.0844) (17/5,30*0.0854) (18/5,30*0.0864) (19/5,30*0.0869)
(20/5,30*0.0870)};
\end{tikzpicture}
\end{center}
\caption{\small Plot of the function $I_{\delta}(x)$ for $\delta=6$.}
\label{fig:functionI}
\end{figure}
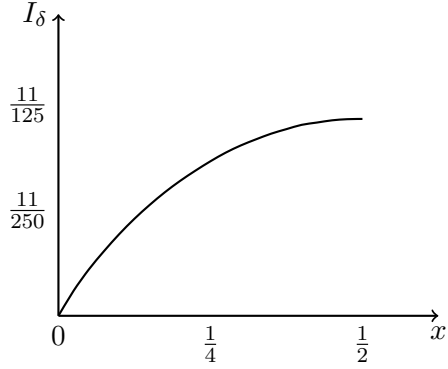

\begin{theorem}
\label{thm:Lbound}
Define
\begin{equation}\label{eq-defmtilde}
\tilde{m} = \min\left\lbrace 1 \leq m \leq n\colon\, \ell_m \geq \tfrac12\ell_n\right\rbrace.
\end{equation}
Then, w.h.p.,
\begin{equation}
\Gamma_n^\star \geq \Gamma_n^-, \qquad
\Gamma_n^- = J\,d_\mathrm{ave}\,I_{d_\mathrm{ave}}
\left(\tfrac12\right) n - h\tilde{m} - o(n).
\label{eq:lboundGamma}
\end{equation}
\end{theorem}

\begin{corollary}
Under hypothesis {\rm (H)} (or the weaker version of {\rm (H)} introduced in Section~\ref{althyp}),
Theorems~{\rm \ref{thm:Ubound}--\ref{thm:Lbound}} yield the following bounds on the crossover
time (see Dommers~\cite[Proposition~2.4]{Dpr}):
\begin{equation}
\label{timesandwich}
\lim_{\beta\to\infty} P_{\boxminus}^{G,\beta}\left(e^{\Gamma_n^- -\varepsilon}
\leq \tau_\boxplus \leq e^{\Gamma_n^+ +\varepsilon}\right) =1.
\end{equation}
\end{corollary}

\noindent
In Corollary~\ref{cor:mvalues} we compute $\bar{m}$, $\ell_{\bar{m}}$, $\tilde{m}$
for two degree distributions: Dirac distributions and power-law distributions.
It is clear that $\tilde{m}=\lceil \tfrac12 n\rceil$ for Dirac distributions.

The bounds we have found in Theorems~\ref{thm:Ubound}--\ref{thm:Lbound} are tight
in the limit of large degrees. Indeed, by the law of large numbers we have that
\begin{equation}
\ell_{n}\frac{\ell_{\bar{m}}}{\ell_{n}}
\left(1-\frac{\ell_{\bar{m}}}{\ell_{n}}\right)
\leq\tfrac{1}{4}\ell_{n}=\tfrac{1}{4} d_{\mathrm{ave}}\,n\left[1+o(1)\right].
\end{equation}
Hence
\begin{equation}
\frac{\Gamma_n^+}{\Gamma_n^-}
= \frac{\frac{1}{4}d_{\mathrm{ave}}\left[1+o\left(1\right)\right]
-\frac{h}{J}\frac{\bar{m}}{n}+o(1)}{d_{\mathrm{ave}}
I_{d_{\mathrm{ave}}}\left(\frac{1}{2}\right)-\frac{h}{J}\frac{\tilde{m}}{n}
-o(1)}.
\label{eq:ratioofbounds}
\end{equation}
In the limit as $d_{\mathrm{ave}}\to\infty$ we have $I_{d_{\mathrm{ave}}}
\left(\frac{1}{2}\right)\to\frac{1}{4}$, in which case \eqref{eq:ratioofbounds}
tends to 1.


\subsection{Discussion}
\label{discussion}

We close this introduction by discussing our main results.

\medskip\noindent
{\bf 1.}
We believe that Theorem~\ref{thm:hyp} holds as soon as
\begin{equation}
0 < h < (d_\mathrm{min}-1)J,
\end{equation}
i.e., we believe that in the limit as $\beta\to\infty$ followed by $n\to\infty$ this choice
of parameters corresponds to the \emph{metastable regime} of our dynamics, i.e.,
the regime where $(\boxminus,\boxplus)$ is a \emph{metastable pair} in the sense
of \cite[Chapter 8]{BdH15}.

\medskip\noindent
{\bf 2.}
The scaling behaviour of $\Gamma_n^\star$ as $n\to\infty$, as well as the geometry
of $\cC_n^\star$ are hard to capture. We can only offer some conjectures.

\begin{conjecture}
\label{conj:scalnucltime}
There exists a $\gamma^\star \in (0,\infty)$ such that
\begin{equation}
\lim_{n\to\infty} \PP_n\Big( \big| n^{-1} \Gamma^\star_n - \gamma^\star\big|
> \delta\Big) = 0 \qquad \forall\,\delta>0.
\end{equation}
\end{conjecture}

\begin{conjecture}
\label{conj:scalcritdrop}
There exists a $c^\star \in (0,1)$ such that
\begin{equation}
\lim_{n\to\infty} \PP_n\Big( \big| n^{-1} \log|\cC^\star_n| - c^\star\big|
> \delta\Big) = 0 \qquad \forall\,\delta>0.
\end{equation}
\end{conjecture}

\begin{conjecture}
\label{conj:scalprefac}
There exists a $\kappa^\star \in (1,\infty)$ such that
\begin{equation}
\lim_{n\to\infty} \PP_n\Big( \big| |\cC^\star_n|\,K^\star_n - \kappa^\star\big|
> \delta\Big) = 0 \qquad \forall\,\delta>0.
\end{equation}
\end{conjecture}

\noindent
As is clear from the results mentioned in Section~\ref{literature}, all three conjectures
are true for the torus, the hypercube and the complete graph. This supports our belief that
they should be true for a large class of random graphs as well.

\medskip\noindent
{\bf 3.}
In Section~\ref{CMdynamic} we will give a \emph{dynamical construction} of $\CM_n$
in which vertices are added one at a time and edges are relocated. This leads to a
\emph{random graph process} $(\CM_n)_{n\in\N}$ whose marginals respect the law
of the Configuration Model. In Section~\ref{tailprop} we will show that this process is
\emph{tail trivial}, i.e., all events in the tail sigma-algebra
\begin{equation}
\cT = \cap_{N \in \N}\, \sigma\left(\cup_{n \geq N} \CM_n\right)
\end{equation}
have probablity 0 or 1. Consequently, the associated communcation height process
$(\Gamma^\star_n)_{n\in\N}$ with $\Gamma^\star_n = \Gamma^\star(\CM_n)$
is tail trivial as well. In particular, both $\gamma^*_- = \liminf_{n\to\infty} n^{-1}
\Gamma^\star_n$ and $\gamma^*_+ = \limsup_{n\to\infty} n^{-1} \Gamma^\star_n$
exists and are constant a.s. Theorems~\ref{thm:Ubound}--\ref{thm:Lbound} show
that $0<\gamma^*_- \leq \gamma^*_+<\infty$. Settling Conjecture~\ref{conj:scalnucltime}
amounts to showing that $\gamma^*_-=\gamma^*_+$.

\medskip\noindent
{\bf 4.}
It was shown by Dommers~\cite{Dpr} that for the Configuration Model with $f=\delta_r$,
$r \in \N\backslash\{1,2\}$, i.e., for a random regular graph with degree $r$, there exist
constants $0<\gamma_-^\star(r)<\gamma_+^\star(r)<\infty$ such that
\begin{equation}
\label{eq:sandwich}
\lim_{n\to\infty} \lim_{\beta\to\infty}
\EE_n\left( P^{\CM_n}_\boxminus\left( e^{\beta n\gamma_-^\star(r)}
\leq \tau_\boxplus \leq e^{\beta n\gamma_+^\star(r)}\right) \right) = 1,
\end{equation}
provided $\frac{h}{J} \in (0,C_0\sqrt{r})$ for some constant $C_0 \in (0,\infty)$ that is small
enough. Moreover, there exist constants $C_1 \in (0,\tfrac14\sqrt{3})$ and $C_2 \in
(0,\infty)$ (depending on $C_0$) such that
\begin{equation}
\gamma_-^\star(r) \geq \tfrac14 Jr - C_1J\sqrt{r}, \qquad
\gamma_+^\star(r) \leq \tfrac14 Jr + C_2J\sqrt{r}, \qquad r \in \N\backslash\{1,2\}.
\end{equation}
The result
in \eqref{eq:sandwich} is derived without hypothesis (H), but it is shown that hypothesis (H)
holds as soon as $r \geq 6$.

\medskip\noindent
{\bf Outline.}
The rest of the paper is organised as follows. In Section~\ref{metstate} we prove that hypothesis
(H) holds under certain constraints on the magnetic field $h$ and the minimal degree of the graph
$d_\mathrm{min}$. Section~\ref{althyp} gives an alternative to hypothesis (H), which holds for a
broader range parameters, yet still permits us to claim our bounds on the crossover time. In
Section~\ref{thmproofs} we prove our upper and lower bounds on $\Gamma_n^\star$. Part of
this proof depends on a \emph{dynamical construction} of $\CM_n$. In Section~\ref{tailprop} we
derive certain properties of this construction.


\section{Proof of Theorem~\ref{thm:hyp}}
\label{metstate}

This section gives a proof of hypothesis (H). We start with the following remark about
the configurations in $\Omega$.

\begin{remark}
{\rm A natural isomorphism between configurations and subsets of vertices of the
underlying graph $G=(V,E)$ comes from identifying $\xi\in\Omega$ with the set
$\{v \in V\colon\,\xi(v)=+1\}$. With this in mind, we denote by $\overline{\xi}$ the
configuration corresponding to the complement of this set: $\{v\in V\colon\,\xi(v)=-1\}$.
Furthermore, for $\zeta,\sigma\in\Omega$ we denote by $E(\zeta,\sigma) \subseteq E$
the set of all unoriented edges $\{(v,w)\in E\colon\,\zeta(v)=\sigma(w)=+1\}$. The main
use of the last definition will be for $\sigma=\overline{\zeta}$: $E(\zeta,\overline{\zeta})$
is the edge boundary of the set $\{v\in V\colon\,\zeta(v)=+1\}$.}
\end{remark}

We next give two lemmas that will be useful later on.

\begin{lemma}
\label{lem:Ideltabound}
For all $\delta \geq 2$ and $0<x\leq\tfrac12$, $I_{\delta}\left(x\right) \leq \left(1-x\right)
-\left(1-x\right)^{2\left(1-1/\delta\right)}$.
\end{lemma}

\begin{proof}
The claim can be verified numerically. For $w \in \left(0,\frac{1}{2}\right]$, let $\tilde{y}
=\left(1-x\right)-\left(1-x\right)^{2\left(1-w\right)}$. Fig.~\ref{fig:cplot} gives a contour 
plot of the function
\begin{equation}
\tilde{I}\left(x,w\right) = x^{x\left(1-w\right)}\left(1-x\right)^{\left(1-x\right)
\left(1-w\right)}\left(1-x-\tilde{y}\right)^{-\left(1-x-\tilde{y}\right)/2}
\left(x-\tilde{y}\right)^{-\left(x-\tilde{y}\right)/2}\tilde{y}^{-\tilde{y}}.
\end{equation}

\begin{figure} [htbp]
\begin{center}
\includegraphics[scale = 0.3]{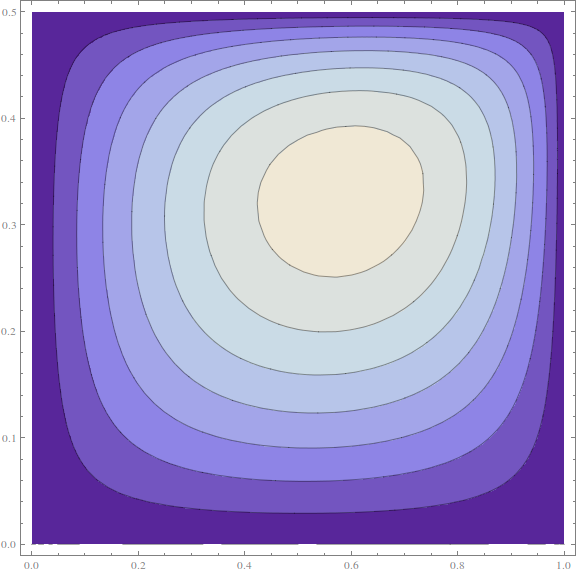}
\end{center}
\caption{\small
A contour plot of $\tilde{I}\left(x,w\right)$ for $x\in \left(0,1\right)$ and 
$w \in \left(0,\frac{1}{2}\right]$. A lighter colour indicates a larger value.}
\label{fig:cplot}
\end{figure}

\noindent 
Note that $\tilde{I}\left(x,w\right)\geq 1$, which immediately implies 
Lemma~\ref{lem:Ideltabound} when we take $w=1/\delta$. It is easy 
to verify that the boundary values corresponding to $x \downarrow 0$ 
and $x \uparrow 1$ result in $\tilde{I}\left(x,w\right) \downarrow 1$.
\end{proof}

\begin{lemma}
\label{lem:Ideltamonotone}
The function $x\rightarrow\frac{I_{\delta}\left(x\right)}{x}$ is non-increasing on
$\left(0,\frac{1}{2}\right]$.
\end{lemma}

\begin{proof}
By definition of $I_{\delta}\left(x\right)$, the function
\begin{equation}
\hat{I}\left(x,z\right) = x^{x\left(1-1/\delta\right)}\left(1-x\right)^{\left(1-x\right)
\left(1-1/\delta\right)}\left(1-x-xz\right)^{-\left(1-x-xz\right)/2}
\left(x-xz\right)^{-\left(x-xz\right)/2}\left(xz\right)^{-xz}
\end{equation}
satisfies $\hat{I}(x,\frac{I_{\delta}\left(x\right)}{x})=1$ for all $x\in\left(0,\frac{1}{2}\right]$.
It will therefore suffice to show that
\begin{equation}
\label{Itask}
\frac{\partial}{\partial x}\vert_{z=\frac{I_{\delta}\left(x\right)}{x}} \hat{I}\left(x,z\right) \geq 0,
\end{equation}
since this implies that, for $\epsilon$ sufficiently small, $\hat{I}(x+\epsilon,\frac{I_{\delta}
\left(x\right)}{x})\geq1$, and hence that
\begin{equation}
\inf\left\{w\colon\,\hat{I}\left(x+\epsilon,w\right)\geq1\right\} \leq\frac{I_{\delta}\left(x\right)}{x},
\end{equation}
and thus $\frac{I_{\delta}(x+\epsilon)}{x+\epsilon} \leq \frac{I_{\delta}\left(x\right)}{x}$.
Observe that
\begin{equation}
\begin{aligned}
&\frac{\partial}{\partial x}\hat{I}\left(x,z\right)\\
&= \hat{I}\left(x,z\right)
\left\{ \log\left(\left(\frac{x}{1-x}\right)^{\left(1-1/\delta\right)}
\left(1-x-xz\right)^{\left(1+z\right)/2}\left(x-xz\right)^{-\left(1-z\right)/2}
\left(xz\right)^{-z}\right)\right\}.
\end{aligned}
\end{equation}
For $z=\frac{I_{\delta}\left(x\right)}{x}$, $\hat{I}\left(x,z\right)=1$ implies
\begin{equation}
\left(x^{x\left(1-1/\delta\right)}\left(1-x\right)^{\left(1-x\right)
\left(1-1/\delta\right)}\right)^{\frac{1}{x}}=\left(1-x-xz\right)^{\frac{1}{2x}
-\left(1+z\right)/2}\left(x-xz\right)^{\left(1-z\right)/2}\left(xz\right)^{z}
\end{equation}
and hence
\begin{equation}
\begin{aligned}
&\frac{\partial}{\partial x}\hat{I}\left(x,z\right)\\
&=\hat{I}\left(x,z\right)\left\{ \log\left(\left(\frac{x}{1-x}\right)^{\left(1-1/\delta\right)}
\left(x^{x\left(1-1/\delta\right)}\left(1-x\right)^{\left(1-x\right)
\left(1-1/\delta\right)}\right)^{-\frac{1}{x}}\left(1-x-xz\right)^{\frac{1}{2x}}\right)\right\}.
\label{eq:derivIhat}
\end{aligned}
\end{equation}
The term inside the logarithm in (\ref{eq:derivIhat}) simplifies to $\left(1-x\right)^{-\frac{1}{x}
\left(1-1/\delta\right)}\left(1-x-xz\right)^{\frac{1}{2x}}$, which satisfies
$\left(1-x\right)^{-\frac{1}{x}\left(1-1/\delta\right)}\left(1-x-xz\right)^{\frac{1}{2x}}\geq1$
whenever $1-x-\left(1-x\right)^{2\left(1-1/\delta\right)} \geq xz=I_{\delta}\left(x\right)$.
By Lemma \ref{lem:Ideltabound}, this is true for all $x\in\left(0,\frac{1}{2}\right]$, and
so \eqref{Itask} follows.
\end{proof}

We can now proceed with the proof of hypothesis (H). Let $\sigma\in\Omega$ be any
configuration that satisfies $x=\ell_{\sigma}/\ell_n\leq\tfrac12$, where $\ell_\sigma
=\sum_{i\in\sigma} d_i$. We will construct a path from $\sigma$ to some $\sigma\prime
\in\Omega$ satisfying $\mathcal{H}\left(\sigma\prime\right)<\mathcal{H}\left(\sigma\right)$
by removing one vertex at a time, obtaining a path $\sigma=\sigma_{0},\ldots,\sigma_{m}
=\sigma\prime$. In particular, at step $t$ we remove any vertex $v_{t}\in\sigma_{t-1}$ that
minimises the quantity $\left|E\left(v_{t},\sigma_{t-1}\backslash v_{t}\right)\right|
-\left|E\left(v_{t},\overline{\sigma_{t-1}}\right)\right|$. It will follow that for every $\sigma_{i}$
in this path, we have $\left|\mathcal{H}\left(\sigma_{i}\right)-\mathcal{H}\left(\sigma_{0}
\right)\right|<\Gamma^{\star}$, which proves the claim of the theorem.

The probability that some configuration $\sigma$, chosen uniformly from all configurations
in $\Omega$ with $\ell_\sigma=L$, has a boundary of size $\left|E\left(\sigma,
\overline{\sigma}\right)\right|=K$ equals
\begin{equation}
\begin{aligned}
&{L \choose K}K!\left(L-K-1\right)!!{\ell_{n}-L \choose K}
\left(\ell_{n}-L-K-1\right)!!/\left(\ell_{n}-1\right)!! \\
&\approx \frac{\left(L\right)!}{\left(L-K\right)!!}
\frac{\left(\ell_{n}-L\right)!}{K!\left(\ell_{n}-L-K\right)!!}\frac{1}{\ell_{n}!!} \\
&\approx  L^L\left(\ell_{n}-L\right)^{\left(\ell_{n}-L\right)}
\left(\ell_{n}-L-K\right)^{-\left(\ell_{n}-L-K\right)/2}
\left(L-K\right)^{-\left(L-K\right)/2}K^{-K}\ell_{n}^{-\ell_{n}/2},
\end{aligned}
\end{equation}
where the symbol $\approx$ stands for equality up to polynomial terms (here of order
$O(n^{2})$). Let $x=L/\ell_{n}$ and $y=K/\ell_{n}$, so that the above
expression becomes
\begin{equation}
\exp\left[\ell_{n}\log\left(x{}^{x}\left(1-x\right)^{\left(1-x\right)}
\left(1-x-y\right)^{-\left(1-x-y\right)/2}\left(x-y\right)^{-\left(x-y\right)/2}y^{-y}\right)\right].
\end{equation}
Furthermore, if we define $\eta\left(x\right)$ by
\begin{equation}
\label{eq:defeta}
\exp\left[\ell_{n}\log\eta\left(x\right)\right]
=\left|\left\{ U\subseteq V\colon\,\ell_U=\ell_{n}x\right\} \right|,
\end{equation}
then the probability of there being any configuration of total degree $L$ having
a boundary size $K$ is bounded from above by
\begin{equation}\label{sd-expbound}
\exp\left[\ell_{n}\log\left(\eta\left(x\right)x{}^{x}\left(1-x\right)^{\left(1-x\right)}
\left(1-x-y\right)^{-\left(1-x-y\right)/2}\left(x-y\right)^{-\left(x-y\right)/2}y^{-y}\right)\right].
\end{equation}

It is easy to see that, by using $\delta=d_{\min}$, the cardinality in the right-hand side of
\eqref{eq:defeta} is bounded from above by ${\ell_{n}/\delta \choose x\ell_{n}/\delta}$.
Using Stirling's approximation for this term, and substituting in \eqref{sd-expbound},
we get
\begin{equation}
\label{eq:probineq}
\begin{aligned}
&\mathbb{P}\left[\exists\, A\subseteq V\colon\,\ell_{A}=x\ell_{n}
\mbox{ and }\left|E\left(A,\overline{A}\right)\right|=y\ell_{n}\right]\\
&\leq \exp\left[\ell_{n}\log\left(x{}^{x\left(1-1/\delta\right)}
\left(1-x\right)^{\left(1-x\right)\left(1-1/\delta\right)}
\left(1-x-y\right)^{-\left(1-x-y\right)/2}\left(x-y\right)^{-\left(x-y\right)/2}y^{-y}\right)\right].
\end{aligned}
\end{equation}
Recall the definition of $I_{\delta}$ from \eqref{eq:functionI} and note that \eqref{eq:probineq}
is exponentially small for $y<I_{\delta}\left(x\right)$, and by a union bound it is exponentially
small for all such $y$.

Suppose that after $s$ vertices have been removed, we reach a configuration
$\sigma_{s}$ with $\mathcal{H}\left(\sigma_{s}\right) < \mathcal{H}\left(\sigma\right)$,
such that for every vertex $v\in\sigma_{s}$ we have
\begin{equation}
\left|E\left(v,\sigma_{s}\backslash v\right)\right|+\frac{h}{J}
>\left|E\left(v,\overline{\sigma_{s}}\right)\right|.
\label{eq:afters}
\end{equation}
In other words, equation (\ref{eq:afters}) states that after removing $s$ vertices we
are at a configuration of lower energy, and removing any additional vertex leads to
a configuration of higher energy. Note that if no such $s$ exists, then we keep on
removing vertices until $\boxminus$ has been reached. By the assumption that
$h$ is sufficiently small (by \eqref{eq:lboundGamma}, it would suffice if $h<\tfrac12\,
J\,d_\mathrm{ave}\,I_{d_\mathrm{ave}} \left(\tfrac12\right) \frac{n}{\tilde{m}}$, where
$\tilde{m}$ was also defined in the aforementioned equation), w.h.p., every configuration
$\sigma$ of total degree $\ell_{\sigma}\leq\tfrac12 \ell_n$ satisfies $\mathcal{H}
\left(\sigma\right)>\mathcal{H}\left(\boxminus\right)$. If $v\in\sigma_{s}$ has no self-loops,
then we have
\begin{equation}
\left|E\left(v,\sigma_{s}\backslash v\right)\right|
=d_{v}-\left|E\left(v,\overline{\sigma_{s}}\right)\right|
\end{equation}
and thus the condition in \eqref{eq:afters} is satisfied when for all $v\in\sigma_{s}$,
\begin{equation}
\label{eq:}
\tfrac12\left(d_{v}+\frac{h}{J} \right) > \left|E\left(v,\overline{\sigma_{s}}\right)\right|.
\end{equation}
The total number of vertices with self-loops is w.h.p.\ of order $o(n)$, and so it will
be evident from the bounds below that this assumption is immaterial. The second
inequality in
\begin{equation}
\left|E\left(\sigma_{s},\overline{\sigma_{s}}\right)\right|
< \tfrac12 \sum_{v\in\sigma_{s}} \left( d_{v}+\frac{h}{J}  \right)
= \tfrac12 \bigg(x\ell_{n}-\sum_{i=1}^{s}d_{i}+\frac{h\left(\left|\sigma\right| -s\right)}{J}  \bigg)
\leq\left|E\left(\sigma,\overline{\sigma}\right)\right|
\end{equation}
holds whenever
\begin{equation}
x\ell_{n}-2\left|E\left(\sigma,\overline{\sigma}\right)\right|+\frac{h\left(\left|\sigma\right| -s\right)}{J}
\leq \sum_{i=1}^{s} d_{i},
\end{equation}
which in particular is true when we take the smallest $s$ such that
\begin{equation}
\sum_{i=1}^{s} d_{i} \geq x\ell_{n}-2I_{\delta}\left(x\right)\ell_{n}
+\frac{h\left(\left|\sigma\right| -s\right)}{J}.
\label{eq:condits}
\end{equation}
Furthermore, by removing $s$ vertices, the change in the size of the boundary at step $t$ is
given by
\begin{equation}
\begin{aligned}
&\left|E\left(\sigma_{t},\overline{\sigma_{t}}\right)\right|
-\left|E\left(\sigma,\overline{\sigma}\right)\right|
= \sum_{i=1}^{t}\big(\left|E\left(v_{i},\sigma_{i-1}\backslash v_{i}\right)\right|
-\left|E\left(v_{i},\overline{\sigma_{i-1}}\right)\right|\big)\\
&\qquad \leq \sum_{i=1}^{t}\left(d_{i}-2\left\lceil d_{i}\,
\frac{I_{\delta}\left(x\right)}{x}\right\rceil \right)
\leq \left(1-2\,\frac{I_{\delta}\left(x\right)}{x}\right)\sum_{i=1}^{t} d_{i}.
\label{eq:stepssum}
\end{aligned}
\end{equation}
The first inequality in \eqref{eq:stepssum} follows from the following observation: note that
w.h.p.\ $|E(\sigma,\overline{\sigma})|$ $\geq I_{\delta}(x)\ell_n$, and hence the ``proportion''
of the total degree of $\sigma$ that is paired with vertices in $\overline{\sigma}$ is at least
$I_{\delta}(x)/x$. This implies that there must be some vertex $v_i$ with a proportion of at
least $I_{\delta}(x)/x$ of its degree connected with vertices in $\overline{\sigma}$. In other
words, $v_i$ shares at least $\lceil d_i \frac{I_{\delta}(x)}{x}\rceil$ edges with $\overline{\sigma}$.

By the definition of $s$, we that $|E(\sigma,\overline{\sigma})|=|E(\sigma_s,\overline{\sigma_s})|
+o(n)$ (when $d_s=o(n)$), and hence dropping the $o(n)$-term is of no consequence in
the following computations. This implies that if $t$ is such that $|E(\sigma_t,\overline{\sigma_t})|
-|E(\sigma,\overline{\sigma})|$ is maximised, then we get (again, possibly after dropping a term of
order $o(n)$)
\begin{equation}
\begin{aligned}
\sum_{i=1}^t \big(\left|E\left(v_{i},\sigma_{i-1}\right)\right|
-\left|E\left(v_{i},\overline{\sigma_{i-1}}\right)\right|\big)
& = & \sum_{i=t+1}^s \big(\left|E\left(v_{i},\overline{\sigma_{i-1}}\right)\right|
-\left|E\left(v_{i},\sigma_{i-1}\right)\right|\big).
\label{eq:peakequality}
\end{aligned}
\end{equation}
Let $m_{t}$ denote the left-hand side of $\eqref{eq:peakequality}$, so that
\begin{equation}
\begin{aligned}
m_{t} &= \sum_{i=1}^{t}\left(d_{i}-2\left|E\left(v_{i},\overline{\sigma_{i-1}}\right)\right|\right)
=  \sum_{i=t+1}^{s}\left(d_{i}-2\left|E\left(v_{i},\sigma_{i-1}\right)\right|\right)\\
&=  \sum_{i=1}^{s}\left(d_{i}-2\left|E\left(v_{i},\sigma_{i-1}\right)\right|\right)
-\sum_{i=1}^{t}\left(d_{i}-2\left|E\left(v_{i},\sigma_{i-1}\right)\right|\right).
\end{aligned}
\end{equation}
Hence
\begin{equation}
\sum_{i=1}^{t}d_{i}\left(\frac{\sum_{i=1}^{t}\left(d_{i}
-2\left|E\left(v_{i},\overline{\sigma_{i-1}}\right)\right|\right)}{\sum_{i=1}^{t}d_{i}}+1\right)
=\sum_{i=1}^{s}d_{i}-\sum_{i=t+1}^{s}2\left|E\left(v_{i},\sigma_{i-1}\right)\right|,
\end{equation}
and thus
\begin{equation}
\begin{aligned}
m_{t} &=  \left(\frac{\sum_{i=1}^{t}\left(d_{i}
-2\left|E\left(v_{i},\overline{\sigma_{i-1}}\right)\right|\right)}
{\sum_{i=1}^{t}d_{i}}\right)\left(\frac{\sum_{i=1}^{t}\left(d_{i}
-2\left|E\left(v_{i},\overline{\sigma_{i-1}}\right)\right|\right)}
{\sum_{i=1}^{t}d_{i}}+1\right)^{-1}\\
&\qquad\qquad\qquad \times \left(\sum_{i=1}^{s}d_{i}
-\sum_{i=t+1}^{s}2\left|E\left(v_{i},\sigma_{i-1}\right)\right|\right)\\
&\leq \frac{1}{2}\left(1-2\frac{I_{\delta}\left(x\right)}{x}\right)
\left(1-\frac{I_{\delta}\left(x\right)}{x}\right)^{-1}\left(x\ell_{n}
-2I_{\delta}\left(x\right)\ell_{n}+\frac{h\left(\left|\sigma\right|-s\right)}{J}\right),
\end{aligned}
\end{equation}
where for the last inequality we use \eqref{eq:condits}--\eqref{eq:stepssum} and
the monotonicity of $y\rightarrow y(y+1)^{-1}$. From \eqref{eq:lboundGamma},
using the fact that $nd_{\mathrm{ave}}=\ell_{n}+o\left(n\right)$, we get that
$ \mathcal{H}\left(\sigma_{t}\right) - \mathcal{H}\left(\sigma\right) < \Gamma^{\star}$
whenever
\begin{equation}
\frac{h}{J\ell_{n}}\left(2\tilde{m}+t+\left(\left|\sigma\right|
-s\right)\left(\frac{x-2I_{\delta}\left(x\right)}{x-I_{\delta}\left(x\right)}\right)\right)
<2I_{d_{\mathrm{ave}}}\left(\tfrac{1}{2}\right)
-\left(x-2I_{\delta}\left(x\right)\right)^{2}\left(x-I_{\delta}\left(x\right)\right)^{-1}.
\label{eq:ihalfbound}
\end{equation}
Note that if $x<2I_{\delta}\left(x\right)$, then for sufficiently small $h$ we can find a
monotone downhill path to $\boxminus$. More precisely, $x<2I_{\delta}\left(x\right)$
implies that the terms in the right-hand side of \eqref{eq:stepssum} become negative,
and hence for $\frac{h}{J}<d_{\min}(\frac{2I_{\delta}\left(x\right)}{x}-1)$ every step in
our path is a downhill step. For $x\geq2I_{\delta}\left(x\right)$, observe first that since
the function $u\rightarrow\frac{\left(1-2u\right)^{2}}{\left(1-u\right)}$ is non-increasing
for $u\leq\frac{1}{2}$, by Lemma \ref{lem:Ideltamonotone}
\begin{equation}
\frac{\left(x-2I_{\delta}\left(x\right)\right)^{2}}{\left(x-I_{\delta}\left(x\right)\right)}
= x\frac{(1-2\frac{I_{\delta}\left(x\right)}{x})^{2}}{(1-\frac{I_{\delta}\left(x\right)}{x})}
\leq \tfrac{1}{2}\frac{(1-4I_{\delta}(\tfrac{1}{2}))^{2}}
{(1-2I_{\delta}\left(\tfrac{1}{2}\right))},
\end{equation}
and thus a sufficient condition for \eqref{eq:ihalfbound} to hold is
\begin{equation}
\tfrac{h}{J}\left(\tfrac{1}{d_{\mathrm{ave}}} + \tfrac{1}{2} \right)
< 2I_{d_{\mathrm{ave}}}\left(\tfrac{1}{2}\right)
-\tfrac{1}{2}\left(1-4I_{d_{\mathrm{min}}}\left(\tfrac{1}{2}\right)\right)^{2}
\left(1-2I_{d_{\mathrm{min}}}\left(\tfrac{1}{2}\right)\right)^{-1}.
\label{eq:Hcond1}
\end{equation}
Hence we have a path $\sigma \to \sigma_{s}$ (or, when such an $s$ satisfying \eqref{eq:afters}
does not exist, a path $\sigma \to \boxminus$) with $\mathcal{H}\left(\sigma_{s}\right)
< \mathcal{H}\left(\sigma\right)$ that never exceeds $\mathcal{H}\left(\sigma\right)$
by $\Gamma^{\star}$ or more, whenever $h$ is sufficiently small and \eqref{eq:ihalfbound}
holds. This proves the claim of the theorem for all configurations $\sigma$ with $\ell_{\sigma}
\leq \tfrac12\ell_n$.

Note also that, for $\ell_{\sigma} > \tfrac12\ell_n$, the same argument can be repeated by
adding a vertex at each step, which will also come at a lower cost since at each step the
magnetisation changes by $-h$.


\section{An alternative to hypothesis (H)}
\label{althyp}

In this section gives a weaker version of hypothesis (H), which nonetheless suffices
as a prerequisite for Theorem~\ref{thm:nucltime}. This weaker version can be verified
for a parameter range that is larger than the one needed in Section~\ref{metstate}.

We can repeat the arguments given in Section~\ref{metstate}. But, instead of insisting
that $\mathscr{V}_{\sigma}<\Gamma^{\star}$ for every configuration $\sigma \in \Omega$,
we require that $\mathscr{V}_{\sigma}$ is bounded from above by our upper bound on
$\Gamma^{\star}$, since this guarantees that our upper bound on the crossover time
is still valid and \eqref{timesandwich} still holds (see Dommers~\cite[Lemma~5.3]{Dpr}).
Thus, it follows from the arguments leading to \eqref{eq:ihalfbound} that we only need
the condition
\begin{equation}
\frac{h}{J\ell_{n}}\left(\bar{m}+t+\left|\sigma\right|
-s\left(\frac{x-2I_{\delta}\left(x\right)}{x-I_{\delta}\left(x\right)}\right)\right)
\leq 2\frac{\ell_{\bar{m}}}{\ell_n}\Big(1-\frac{\ell_{\bar{m}}}{\ell_n}\Big)
-\left(x-2I_{\delta}\left(x\right)\right)^{2}\left(x-I_{\delta}\left(x\right)\right)^{-1}.
\label{eq:ihalfbound2}
\end{equation}
For $h$ sufficiently small, the ratio $\frac{\ell_{\bar{m}}}{\ell_n}$ can be made arbitrarily
close to $\tfrac12$, in which case the right-hand side of \eqref{eq:ihalfbound2} becomes
strictly positive. This implies that the inequality in \eqref{eq:ihalfbound2} holds for any
$\delta \geq 3$ whenever $h$ is sufficiently small.


\section{Proof of Theorems~\ref{thm:Ubound} and \ref{thm:Lbound}}
\label{thmproofs}


\subsection{A dynamic construction of the configuration model}
\label{CMdynamic}

Prior to giving the proof of Theorems~\ref{thm:Lbound} and \ref{thm:Ubound}, we introduce
a \emph{dynamical construction} of the CM graph. This will be used to obtain the upper
bound in Theorem~\ref{thm:Ubound}.

Let $V=\{v_{i}\} _{i=1}^{n}$ be a sequence of vertices with degrees
$\{d_{i}\} _{i=1}^{n}$. In this section we construct a graph $G=(V,E)$
with the same distribution as a graph generated through the Configuration
Model algorithm, but in a \emph{dynamical way}, as follows.

Suppose that $\xi_{m}$ is a uniform random matching of the integers
$\{1,\ldots,2m\} $, denoted by $\xi_{m}=\{(x_{1},x_{2}),\ldots,(x_{2m-1},
x_{2m}\}$, where the pairs are listed in the order they were created
(which is not an important issue, so long as we agree on some labeling).
Next, let $u$ be uniform on $\{1,\ldots,2m,2m+1\} $ and set $\xi_{m+1}
=\xi_{m}\cup\{(2m+2,u)\}$ if $u=2m+1$. Else if $u \neq 2m+1$, then w.l.o.g.\
$u=x_{2i-1}$ for some $i\leq m$, and we set $\xi_{m+1}=\{\xi_{m}\backslash
\{(x_{2i-1},x_{2i})\}\}\cup\{(2m+2,x_{2i-1}),(2m+1,x_{2i})\}$. Then $\xi_{m+1}$
is a uniform matching of the points $\{1,\ldots,2m,2m+2\}$. It is now obvious
how the construction of $G$ follows from the given scheme.


\subsection{Energy estimates}

Label the vertices of the graph so that their degrees satisfy $d_{1}\leq\ldots\leq
d_{n}$. Let $\gamma\colon\,\boxminus\to\boxplus$ be the path that successively
flips the vertices $v_{1},\ldots,v_{n}$ (in that order), and let $\ell_{m}=\sum_{i=1}^{m}
d_{i}$. We show that, w.h.p., for every
$1\leq m\leq n$,
\begin{equation}
\mathcal{H}\left(\gamma_{m}\right)-\mathcal{H}\left(\boxminus\right)
=J\ell_{m}\Big(1-\frac{\ell_{m}}{\ell_{n}}\Big)-mh \pm O\big(\ell_{n}^{3/4}\big).
\label{eq:pathbound}
\end{equation}

We are particularly interested in the maximum of \eqref{eq:pathbound} over all
$1\leq m \leq n$. To this avail, observe that the function defined by
\begin{equation}
g\left(x\right) = Jx\left(1-x\right)-h\left(x\right)
\label{eq:functiong}
\end{equation}
has at most one maximum for $x\in[0,1]$ if $x \mapsto h(x)$ is non-decreasing.
Thus, taking $x=\frac{\ell_{m}}{\ell_{n}}$ and $h(x)=hx\frac{m}{\ell_{m}}$, we
see that our definition of $\bar{m}$ in Theorem~\ref{thm:Ubound} is justified.
Furthermore, note the equivalent conditions
\begin{equation}
\ell_{m}\left(1-\frac{\ell_{m}}{\ell_{n}}\right)
\geq \ell_{m+1}\left(1-\frac{\ell_{m+1}}{\ell_{n}}\right)-\frac{h}{J}
\iff \frac{h}{J} \geq d_{m+1}\left(1-\frac{2\ell_{m}}{\ell_{n}}\right)
-O\left(\frac{d_{m}^{2}}{\ell_{n}}\right),
\label{eq:equivalentconditions}
\end{equation}
with the last term in \eqref{eq:equivalentconditions} disappearing whenever $d_{m+1}
=o\left(\sqrt{\ell_{n}}\right)$. Note that \eqref{eq:equivalentconditions} gives us an alternative
formulation for $\bar{m}$ in the statement of Theorem~\ref{thm:Ubound}, which we will
use to compute $\bar{m}$ below.


\subsection{Two examples}

Two commonly studied degree distributions for the Configuration Model
are the Dirac distribution
\begin{equation}
q_r(k) = \delta_r(k), \qquad k \in \N_0,
\label{eq:Diracdist}
\end{equation}
for some $r \in \N$ (i.e., the $r$-regular graph), and the power-law
distribution
\begin{equation}
q_{\tau,\delta}(k) = \mathbb{P}\left[d_{i}=\delta+k\right]
= \frac{\left(\delta+k\right)^{-\tau}}{\sum_{i\in\N_0}
\left(\delta+i\right)^{-\tau}}, \qquad k \in \N_0,
\label{eq:powerlawdist}
\end{equation}
for some exponent $\tau \in (2,\infty)$ and shift $\delta \in \N$.

For these degree distributions we get the following corollary of
Theorems~\ref{thm:Ubound}--\ref{thm:Lbound}:

\begin{corollary}
\label{cor:mvalues}
(a) For the Dirac-distribution in \eqref{eq:Diracdist},
\begin{equation}
Jr I_r\left(\tfrac12\right)n -\frac{hn}{2}-o(n) \leq \Gamma_{n}^{\star} \leq \frac{Jr}{4}
n \left(1-\left(\frac{h}{Jr}\right)^2\right)  \pm O\big(n^{3/4}\big).
\end{equation}
(b) For the power-law distribution distribution in \eqref{eq:powerlawdist}, $\bar{m}$
and $\ell_{\bar{m}}$ are given by \eqref{eq:mprimeval} and \eqref{eq:lm/ln}.\\
(c) For the power-law distribution distribution in \eqref{eq:powerlawdist}, $\tilde{m}$
is given by \eqref{eq:tildem}.
\end{corollary}

\begin{proof}
(a) Straightforward.\\
(b) $\{d_{i}\} _{i=1}^{n}$ are i.i.d.\ with degree distribution $q_{\tau,\delta}$. Let
$s_{\tau,\delta,k} = \sum_{i=1}^{n} \mathbf{1}\left\{d_{i}\leq \delta + k\right\}$,
and note that
\begin{equation}
\mathbb{E}\left[s_{\tau,\delta,k}\right]
=n\left(1-\frac{\xi_{\tau}\left(\delta+k+1\right)}{\xi_{\tau}\left(\delta\right)}\right)
\end{equation}
with $\xi_{\tau}\left(a\right)=\sum_{i=a}^{\infty}i^{-\tau}$ for $a\geq0$. We claim that,
for $k$ sufficiently small, $s_{\tau,\delta,k}$ is concentrated around its mean. Indeed,
define $a_{\delta,k}=\sum_{i}\mathbf{1}\left\{ d_{i}=\delta+k\right\}$, $k\in\N_0$, and
note that for any i.i.d.\ sequence we have $a_{\delta,k}\stackrel{d}{=}\mathrm{Bin}\left(n,
p_{\delta,k}\right)$, where $p_{\delta,k}=\mathbb{P}\left[d_{i}=\delta+k\right]$. From
Hoeffding's inequality we get that
\begin{equation}
\mathbb{P}\left[\left|a_{\delta,k}-np_{\delta,k}\right|>n^{\frac{1}{2}+\frac{1}{6}}\right]
\leq\exp\left(-2n^{\frac{1}{3}}\right).
\end{equation}
Hence, for any $k=O(n^{1/6})$,
\begin{equation}
\begin{aligned}
\mathbb{P}\left[\left|s_{\tau,\delta,k}-\mathbb{E}
\left[s_{\tau,\delta,k}\right]\right|>n^{\frac{1}{2}+\frac{1}{3}}\right]
&\leq \mathbb{P}\left[\bigcup_{m=0}^{k}
\left|a_{\delta,m}-np_{\delta,m}\right|>n^{\frac{1}{2}+\frac{1}{6}}\right]
\leq n^{\frac{1}{6}}\exp\left(-2n^{\frac{1}{3}}\right).
\label{eq:bnd-s}
\end{aligned}
\end{equation}
Note that if $p_{\delta,k}=q_{\tau,\delta}(k)$, then $\mathbb{E}[a_{\delta,k}]
=n\,\frac{(\delta+k)^{-\tau}}{\xi_{\tau}(\delta)}$, and w.h.p.\ $\ell_{n}=n
\frac{\xi_{\tau-1}(\delta)}{\xi_{\tau}(\delta)}+o(n)$. Hence we define
\begin{equation}
\kappa = \min\left\{k\in\N\colon\,\frac{h}{J}\geq\left(\delta+k-1\right)
\left(1-\left(\frac{\xi_{\tau-1}\left(\delta+k\right)}
{\xi_{\tau-1}\left(\delta\right)}\right)\right)\right\} -1,
\label{eq:defkappa}
\end{equation}
which certainly satisfies $\kappa= o\left(n^{1/6}\right)$. From (\ref{eq:equivalentconditions})
and the monotonicity of (\ref{eq:functiong}) it follows that w.h.p.
\begin{equation}
\begin{aligned}
\bar{m}
&=n\Bigg[\left(1-\frac{\xi_{\tau}\left(\delta+\kappa\right)}{\xi_{\tau}\left(\delta\right)}\right)\\
&+\min\left\{ y\in\left[0,1\right]\colon\,\frac{h}{J}\geq\left(\delta+\kappa\right)
\left(1-\left(\frac{\xi_{\tau-1}\left(\delta+\kappa\right)}{\xi_{\tau-1}\left(\delta\right)}
+\frac{y\kappa\xi_{\tau}\left(\delta\right)}{\xi_{\tau-1}
\left(\delta\right)}\right)\right)\right\} \Bigg]+o\left(n\right)
\label{eq:mprimeval}
\end{aligned}
\end{equation}
and
\begin{equation}
\frac{\ell_{\bar{m}}}{\ell_{n}}
=\left(\frac{\xi_{\tau-1}\left(\delta+\kappa\right)}{\xi_{\tau-1}\left(\delta\right)}
+\frac{y\kappa\xi_{\tau}\left(\delta\right)}{\xi_{\tau-1}\left(\delta\right)}\right)
+o\left(1\right),
\label{eq:lm/ln}
\end{equation}
where $y$ is the taken as the argument of the minimum in (\ref{eq:mprimeval}).
Since we know $\ell_{n}$ up to $o(n)$, this also gives the value of $\ell_{\bar{m}}$.\\
(c) Note that $\tilde{m} = \sum_{i=0}^{\kappa} a_{\delta,i}$ where $\kappa$ is the
least integer such that
\begin{equation}
\delta a_{\delta,0} + (\delta + 1)a_{\delta,0} + \ldots
+ (\delta + \kappa)a_{\delta,\kappa} \geq \tfrac12\ell_n.
\end{equation}
By the concentration results given above, we see that w.h.p.
\begin{equation}
\kappa = \min\left\lbrace m\in\N\colon\,
\frac{\xi_{\tau -1}(\delta) +\xi_{\tau -1}(\delta + m+1) }{\xi_{\tau}(\delta)}
\geq \tfrac12 d_{\mathrm{ave}} \right\rbrace
\end{equation}
and
\begin{equation}
\tilde{m} = \frac{n}{\xi_{\tau}(\delta)} \sum_{i=0}^{\kappa} (\delta + i)^{\tau} \: + o(n).
\label{eq:tildem}
\end{equation}
\end{proof}


\subsection{Proof of Theorem \ref{thm:Ubound}}

\begin{proof}
Consider a sequence of matchings $\left\{\xi_{1},\xi_{2},\ldots,\xi_{M/2}\right\}$
constructed in a dynamical way as outlined above, where $M$ is some even integer.
Let $0\leq x\leq M$ be even, and define $z_{x,0}=x$ and $z_{x,t}=\sum_{i=1}^{x}
\sum_{m=1}^{x}\mathbf{1}\left\{\left(i,m\right)\in\xi_{x/2+t}\right\}$. Then
\begin{equation}
z_{x,t+1}=z_{x,t}-2\mathbf{1}\left\{ Y_{t+1}\right\},
\label{eq:ziplus1}
\end{equation}
where $Y_{t}$ is the event that $x+2t-1$ and $x+2t$ are both paired with terms in
$\left[x\right]$. Note that
\begin{equation}
\mathbb{P}\left[Y_{t+1}\left|\mathscr{G}{}_{x+t}\right.\right]=\frac{z_{x,t}}{x+2+1},
\label{eq:Ytplus1}
\end{equation}
where $\mathscr{G}_{x+t}$ is the $\sigma$-algebra generated by $\left\{\xi_{1},\ldots,
\xi_{x/2+t}\right\}$. Therefore
\begin{equation}
\mathbb{E}\left[z_{x,t+1}\left|\mathscr{G}{}_{x+t}\right.\right]
=  z_{x,t}-\frac{2z_{x,t}}{x+2t+1}
= z_{x,t}\left(\frac{x+2t-1}{x+2t+1}\right)
\end{equation}
and so
\begin{equation}
\mathbb{E}\left[z_{x,t+1}\right]
= z_{0}\prod_{j=0}^{t}\left(\frac{x+2j-1}{x+2j+1}\right)
= x\left(\frac{x-1}{x+2t+1}\right).
\label{eq:Expztplus1}
\end{equation}
To compute the second moment, observe that
\begin{equation}
z_{x,t+1}^{2}=z_{x,t}^{2}-4z_{x,t}\mathbf{1}\left\{ Y_{t+1}\right\}
+4\mathbf{1}\left\{ Y_{t+1}\right\},
\label{eq:zmoment2}
\end{equation}
and so
\begin{equation}
\mathbb{E}\left[z_{x,t+1}^{2}\left|\mathscr{G}{}_{x+t}\right.\right]
= z_{x,t}^{2}-\frac{4z_{x,t}^{2}}{x+2t+1}+\frac{4z_{x,t}}{x+2t+1}
= z_{x,t}^{2}\left(\frac{x+2t-3}{x+2t+1}\right)+\frac{4z_{x,t}}{x+2t+1}.
\end{equation}
 Then
\begin{equation}
\begin{aligned}
&\mathbb{E}\left[z_{x,t+1}^{2}\right]\\
& = \mathbb{E}\left[z_{x,t}^{2}\right]\left(\frac{x+2t-3}{x+2t+1}\right)
+\frac{4\mathbb{E}\left[z_{x,t}\right]}{x+2t+1}\\
&= x^{2}\prod_{i=0}^{t}\left(\frac{x+2i-3}{x+2i+1}\right)+\sum_{i=0}^{t}
\frac{4\mathbb{E}\left[z_{x,i}\right]}{x+2i+1}
\prod_{j=i+1}^{t}\left(\frac{x+2j-3}{x+2j+1}\right)\\
&= \frac{x^{2}\left(x-3\right)\left(x-1\right)}{\left(x+2t+1\right)\left(x+2t-1\right)}
+\sum_{i=0}^{t}\frac{4x\left(x-1\right)}{\left(x+2i+1\right)\left(x+2i-1\right)}
\frac{\left(x+2i-1\right)\left(x+2i+1\right)}{\left(x+2t+1\right)\left(x+2t-1\right)} \\
&= \frac{x^{2}\left(x-3\right)\left(x-1\right)}{\left(x+2t+1\right)\left(x+2t-1\right)}
+\frac{4x\left(x-1\right)\left(t+1\right)}{\left(x+2t+1\right)\left(x+2t-1\right)} \\
&= \frac{x\left(x-1\right)}{\left(x+2t+1\right)\left(x+2t-1\right)}\left(x\left(x-3\right)
+4\left(t+1\right)\right),
\label{eq:expzmoment2}
\end{aligned}
\end{equation}
while
\begin{equation}
\left(\mathbb{E}\left[z_{x,t+1}\right]\right)^{2}=x^{2}\left(\frac{x-1}{x+2t+1}\right)^{2}
\end{equation}
and so
\begin{equation}
\mathbb{E}\left[z_{x,t+1}^{2}\right]-\left(\mathbb{E}\left[z_{x,t+1}\right]\right)^{2}
=\frac{x\left(x-1\right)}{x+2t+1}\left(\frac{x\left(x-3\right)
+4\left(t+1\right)}{x+2t-1}-\frac{x\left(x-1\right)}{x+2t+1}\right).
\end{equation}
It follows that if we let $w_{x,t}=\frac{z_{x,t}}{x+2t}$, $t\geq1$, then
\begin{equation}
\begin{aligned}
&\mathbb{E}\left[w_{x,t+1}^{2}\right]-\left(\mathbb{E}\left[w_{x,t+1}\right]\right)^{2}\\
& =  \frac{\frac{x}{x+2\left(t+1\right)}\frac{\left(x-1\right)}{x+2\left(t+1\right)}}
{x+2t+1}\left(\frac{4\left(t+1\right)}{x+2t-1}+\frac{x-3}{1+2\frac{t}{x}-\frac{1}{x}}
-\frac{x-1}{1+2\frac{t}{x}+\frac{1}{x}}\right)\\
&= \frac{\frac{x}{x+2\left(t+1\right)}\frac{\left(x-1\right)}{x+2\left(t+1\right)}}
{x+2t+1}\left(\frac{4\left(t+1\right)}{x+2t-1}+\frac{-4\frac{t}{x}
-4\frac{1}{x}}{\left(1+2\frac{t}{x}-\frac{1}{x}\right)\left(1+2\frac{t}{x}
+\frac{1}{x}\right)}\right)\\
&= \frac{4\frac{x}{x+2\left(t+1\right)}\frac{\left(x-1\right)}{x+2\left(t+1\right)}}
{x+2t+1}\frac{t+1}{x+2t-1}\left(1-\frac{1}{1+2\frac{t}{x}+\frac{1}{x}}\right).
\label{eq:wvar}
\end{aligned}
\end{equation}

Observe also that for
\begin{equation}
\bar{z}_{x,t} = \sum_{i=1}^{x}\sum_{m=x+1}^{M}
\mathbf{1}\left\{ \left(i,m\right)\in\xi_{x/2+t}\right\} =x-z_{x,t}
\label{eq:zbar}
\end{equation}
and
\begin{equation}
\bar{w}_{x,t} = \frac{\bar{z}_{x,t}}{x+2t}=\frac{x}{x+2t}-w_{x,t}
\end{equation}
the same variance calculations follow, so that $\mathbb{E}\left[\bar{w}_{x,t+1}^{2}\right]
-\left(\mathbb{E}\left[\bar{w}_{x,t+1}\right]\right)^{2}$ is also given by (\ref{eq:wvar}). For
$\alpha\in\left(0,1\right)$ and $1\leq i\leq k\left(\alpha\right)$ with $k\left(\alpha\right)
=M/\left\lfloor M^{\alpha}\right\rfloor $, let $x_{i}=i\left\lfloor M^{\alpha}\right\rfloor $ and
note that
\begin{equation}
\begin{aligned}
&\sum_{i=1}^{k\left(\alpha\right)} \left(\mathbb{E}\left[\bar{w}_{x,t+1}^{2}\right]
-\left(\mathbb{E}\left[\bar{w}_{x,t+1}\right]\right)^{2}\right)\\
&=\frac{1}{2}M^{-2}\left(M-1\right)^{-1}\left(M-3\right)^{-1}\sum_{i=1}^{k\left(\alpha\right)}
x_{i}\left(x_{i}-1\right)\left(M-x_{i}\right)\left(1-\frac{x_{i}}{M-1}\right)
=O\left(M^{-\alpha}\right).
\label{eq:2ndmomentbndforpath}
\end{aligned}
\end{equation}
From Markov's inequality we have that
\begin{equation}
\begin{aligned}
&\mathbb{P}\left[\exists\,i\mbox{ such that }\frac{\left|\bar{z}_{x_{i},\left(M-x_{i}\right)/2}
-\mathbb{E}\left[\bar{z}_{x_{i},\left(M-x_{i}\right)/2}\right]\right|}{M}
>M^{-\frac{\alpha}{3}}\right]\\
&= \mathbb{P}\left[\exists\, i\mbox{ such that }\frac{\left|\bar{z}_{x_{i},\left(M-x_{i}\right)/2}
-\mathbb{E}\left[\bar{z}_{x_{i},\left(M-x_{i}\right)/2}\right]\right|^{2}}{M^{2}}
>M^{-\frac{2\alpha}{3}}\right]\\
&= \mathbb{P}\left[\exists\, i\mbox{ such that } \left|\bar{w}_{x_{i},\left(M-x_{i}\right)/2}
-\mathbb{E}\left[\bar{w}_{x_{i},\left(M-x_{i}\right)/2}\right]\right|^{2}
>M^{-2\alpha/3}\right]= O\left(M^{-\alpha/3}\right)
\end{aligned}
\end{equation}
and thus we have that w.h.p.\ for every $1\leq i\leq k\left(\alpha\right)$,
\begin{equation}
\left|\bar{z}_{x_{i},\left(M-x_{i}\right)/2}-\mathbb{E}\left[\bar{z}_{x_{i},
\left(M-x_{i}\right)/2}\right]\right|=O\left(M^{1-\alpha/3}\right).
\label{eq:bndziconcentration}
\end{equation}
Now suppose that $x_{i}\leq x\leq x_{i+1}$. Then, clearly, $\bar{z}_{x_{i},
\left(M-x_{i}\right)/2}-M^{\alpha}\leq\bar{z}_{x,\left(M-x\right)/2}\leq
\bar{z}_{x_{i},\left(M-x_{i}\right)/2}+M^{\alpha}$, and via (\ref{eq:Expztplus1})
and (\ref{eq:zbar}) we conclude that w.h.p.\ for every $1\leq x\leq M$ we have
\begin{equation}
\label{eq:boundzbarvariance}
\begin{aligned}
&\left|\bar{z}_{x,\left(M-x\right)/2}-\mathbb{E}\left[\bar{z}_{x,\left(M-x\right)/2}\right]\right|\\
& \leq\left|\bar{z}_{x,\left(M-x\right)/2}-\bar{z}_{x_{i},\left(M-x_{i}\right)/2}\right|
+\left|\bar{z}_{x_{i},\left(M-x_{i}\right)/2}
-\mathbb{E}\left[\bar{z}_{x_{i},\left(M-x_{i}\right)/2}\right]\right|\\
&\qquad +\left|\mathbb{E}\left[\bar{z}_{x_{i},\left(M-x_{i}\right)/2}\right]
-\mathbb{E}\left[\bar{z}_{x,\left(M-x\right)/2}\right]\right|\\
& =\left|\bar{z}_{x,\left(M-x\right)/2}-\bar{z}_{x_{i},\left(M-x_{i}\right)/2}\right|
+\left|\bar{z}_{x_{i},\left(M-x_{i}\right)/2}-x_{i}\left(\frac{M-x_{i}}{M-1}\right)\right|\\
&\qquad +\left|x_{i}\left(\frac{M-x_{i}}{M-1}\right)-x\left(\frac{M-x}{M-1}\right)\right|\\
&\leq M^{\alpha}+O\left(M^{1-\alpha/3}\right)+M^{\alpha}.
\end{aligned}
\end{equation}
Now let $\gamma_{s}$ be any configuration on the path $\gamma:\boxminus\rightarrow\boxplus$
defined above. Then w.h.p.
\begin{equation}
\begin{aligned}
\mathcal{H}\left(\gamma_{s}\right)-\mathcal{H}\left(\boxminus\right)
&= J\left|E\left(\gamma_{s},\overline{\gamma_{s}}\right)\right|-hs\\
&= J\bar{z}_{\ell_{\gamma_{s}},\left(\ell_{n}-\ell_{\gamma_{s}}\right)/2}-hs
=J\ell_{\gamma_{s}}\left(1-\frac{\ell_{\gamma_{s}}}{\ell_{n}}\right)-hs+O\left(\ell_{n}^{3/4}\right),
\end{aligned}
\end{equation}
where the last line follows from \eqref{eq:boundzbarvariance} with $x=\ell_{m}$, $M=\ell_{n}$
and $\alpha=\frac{3}{4}$, and uses the fact that $\mathbb{E}\left[\bar{z}_{x,t}\right]
=x\left(1-\frac{x-1}{x+2t+1}\right)$. By definition, this quantity is maximised when
$\ell_{\gamma_{s}}$ is replaced by $\ell_{\bar{m}}$, from which the statement of the
theorem follows.
\end{proof}


\subsection{Proof of Theorem \ref{thm:Lbound}}

\begin{proof}
Setting $x=\tfrac12$ in~\eqref{sd-expbound}, we get that the probability of there being
any configuration of total degree $\ell_n/2$ having a boundary size $y\ell_n$ is bounded
from above by
\begin{equation}
\exp\left[\ell_{n}\log\left(\tfrac12\eta\left(\tfrac12\right)
\left(\tfrac12-y\right)^{-\left(\tfrac12-y\right)}y^{-y}\right)\right].
\label{eq:expbound}
\end{equation}
By the law of large numbers, w.h.p.\ we have that $n=\ell_{n}/d_{\mathrm{ave}}+
o\left(n\right)$. Combining this with $\left|\left\{ U\subseteq V\colon\,\ell_U
=\ell_{n}x\right\} \right|\leq2^n$ we get that $\eta\left(\tfrac12\right) \leq 2^{1/d_{\mathrm{ave}}}$.
It follows that if $y< I_{d_{\mathrm{ave}}}\left(\tfrac12\right)$, then (\ref{eq:expbound})
decays exponentially. Hence, all configurations $\sigma$ with $\ell_\sigma=\ell_n/2$ have,
w.h.p., an energy at least
\begin{equation}
\cH(\sigma) \geq J I_{d_{\mathrm{ave}}}\left(\tfrac12\right) \ell_n - h |\sigma| + \cH(\boxminus),
\end{equation}
and the lower bound on $\Gamma^\star_n$ in~\eqref{eq:lboundGamma} follows.
\end{proof}


\section{Tail properties of the dynamically constructed $\CM_n$}
\label{tailprop}

In this section we explore some properties of the dynamical construction of $\CM_n$
introduced in Section~\ref{CMdynamic}.


\subsection{Trivial tail $\sigma$-algebra}

Let $V_{n}=\left(v_{1},\ldots,v_{n}\right)$ be the vertices with corresponding
degree sequence
\begin{equation}
\vec{d}_{n}=\left(d_{1},\ldots,d_{n}\right),
\end{equation}
and let $G_{n}=\left(V_{n},E_{n}\right)$ and $G_{n}\prime=\left(V_{n},E_{n}\prime
\right)$ be two independent Configuration Models with the same degree sequence
$\vec{d}_{n}$. We will extend $G_{n}$ and $G_{n}\prime$ to larger graphs, $G_{n+t}
=\left(V_{n+t},E_{n+t}\right)$ and $G_{n+t}\prime=\left(V_{n+t},E_{n+t}\prime\right)$,
respectively, with degree sequence
\begin{equation}
\vec{d}_{n}=\left(d_{1},\ldots,d_{n},d_{n+1},\ldots d_{n+t}\right),
\end{equation}
by utilising a pairing scheme similar to the one introduced in Section \ref{CMdynamic}:
\begin{itemize}
\item[$\bullet$]
If $\xi_{m}$ is a uniform random matching of the integers $\left\{ 1,\ldots,2m\right\}$,
denoted by $\xi_{m}=\left\{ \left(x_{1},x_{2}\right),\ldots,\left(x_{2m-1},x_{2m}\right)
\right\} $, then take $u_{1}$ to be uniform on $\left\{ 1,\ldots,2m\right\} $ and
$u_{2}$ to be uniform on $\left\{ 1,\ldots,2m,2m+1\right\} $. If $u_{2}=2m+1$, then
set $\xi_{m+1}=\xi_{m}\cup\left\{ \left(2m+1,2m+2\right)\right\} $. Otherwise add to
$\xi_{m}$ the pairs $\left(2m+1,u_{1}\right)$ and $\left(2m+2,u_{2}\right)$ when
$u_{1}\neq u_{2}$, and when $u_{1}=u_{2}$, only add to $\xi_{m}$ the pair
$\left(2m+2,u_{2}\right)$. In either case, if there are two remaining terms that are
unpaired, then pair them to each other and add this pair to $\xi_{m}$. Needless
to say, we also remove from $\xi_{m}$ old pairs that were undone by the introduction
of $2m+1$ and $2m+2$. Again, this construction leads to $\xi_{m+1}$, a uniform
matching of the points $\left\{1,\ldots,2m+2\right\}$.
\end{itemize}
Now construct the coupled graphs $\left(G_{n+t},G_{n+t}\prime\right)$ by starting
with $\left(G_{n},G_{n}\prime\right)$ and using the same uniform choice
\begin{equation}
\left\{u_{i}\right\} _{i=1}^{|\vec{d}_{n+t}|-|\vec{d}_{n}|}
\end{equation}
to determine new edges in both graphs. Note that, under this scheme, every term
(half-edge) $s>|\vec{d}_{n}|$ is paired with the same term in $G_{n+t}$
as in $G_{n+t}\prime$. In other words, for all $1\leq j\leq|\vec{d}_{n+t}|$
we have $\left(s,j\right)\in E_{n+t}$ if and only if $\left(s,j\right)\in E_{n+t}\prime$.
For $s\leq|\vec{d}_{n}|$ and $1\leq j\leq|\vec{d}_{n+t}|$, we have
\begin{equation}
\begin{aligned}
&\mathbb{P}\left[\mathbf{1}_{\left\{ \left(s,j\right)\in E_{n+t}\right\} }
\neq \mathbf{1}_{\left\{ \left(s,j\right)\in E_{n+t}\prime\right\} }\right]\\
&\leq  \mathbb{P}\left[\bigcap_{i=1}^{|\vec{d}_{n+t}|
-|\vec{d}_{n}|}\left\{ u_{i}\neq s\right\} \right]
= \prod_{i=|\vec{d}_{n}|}^{|\vec{d}_{n+t}|}\left(1-\frac{1}{i-1}\right)
=\frac{|\vec{d}_{n}|-2}{|\vec{d}_{n+t}|-1}.
\end{aligned}
\end{equation}
Hence
\begin{equation}
\mathbb{P}\left[\bigcup_{s=1}^{\left|\vec{d}_{n}\right|}
\left\{ \mathbf{1}_{\left\{ \left(s,j\right)\in E_{n+t}\right\} }
\neq \mathbf{1}_{\left\{ \left(s,j\right)\in E_{n+t}\prime\right\} }\right\} \right]
\leq\frac{|\vec{d}_{n}|(|\vec{d}_{n}|-2)}
{|\vec{d}_{n+t}|-1}.
\end{equation}
Thus, we conclude that $\mathbb{P}\left[G_{n+t}\neq G_{n+t}\prime\right] =
O\left(\frac{1}{t}\right)$.

We can now make the following standard argument to show that the process
above has a trivial tail-sigma-algebra. Let $\mathscr{F}_{t}=\sigma\left(\xi_{1},
\ldots.\xi_{t}\right)$ and $\mathscr{F}_{t}^{+}=\sigma\left(\xi_{t+1},\ldots\right)$.
The tail sigma-algebra is given by $\mathscr{T}=\cap_{n\in\mathbb{N}}
\mathscr{F}_{t}^{+}$. For any $A\in\mathscr{T}$, there is a sequence of events
$A_{1},A_{2},\ldots$ such that
\begin{equation}
\lim_{t\to\infty}\mathbb{P}\left[A_{t}\triangle A\right]=0,
\end{equation}
and hence also
\begin{equation}
\lim_{t\to\infty}\mathbb{P}\left[A_{t}\cap A\right] =\mathbb{P}\left[A\right],
\qquad \lim_{t\to\infty}\mathbb{P}\left[A_{t}\right]=\mathbb{P}\left[A\right].
\end{equation}
But, since $A\in\mathscr{F}_{t}^{+}$ for all $t$, it follows that $\mathbb{P}\left[A_{t}
\cap A\right]=\mathbb{P}\left[A_{t}\right]\mathbb{P}\left[A\right]$, and hence
$\mathbb{P}\left[A\right]=\mathbb{P}\left[A\right]^{2}$. This shows that $\mathscr{T}$
is a trivial sigma-algebra. Therefore, given $\left\{ d_{i}\right\} _{i\in\N}$ (but also
by the law of large numbers for i.i.d.\ sequences),
\begin{equation}
\limsup_{n\to\infty}\frac{\Gamma_{n}^{\star}}{n}= \gamma^*_+,
\qquad \liminf_{n\to\infty}\frac{\Gamma_{n}^{\star}}{n}= \gamma^*_-,
\end{equation}
for some $\gamma^*_+,\gamma^*_- \in\R$ with $\gamma^*_+ \geq \gamma^*_-$.


\subsection{Oscillation bounds}

It is possible to obtain bounds on the possible oscillations of $n\mapsto
\Gamma_{n}^{\star}/n$.

\begin{lemma}
\label{lem: gammadifference}
Let $G=(V,E)$ and $\tilde{G}=(\tilde{V},\tilde{E})$ be two connected graphs.
Suppose that $\left|E\nabla\tilde{E}\right|\leq k$ under some labelling of the
vertices in $V$ and $\tilde{V}$ (i.e., a one-to-one map from $V$ to $\tilde{V}$).
Then
\begin{equation}
|\Gamma^{\star}-\tilde{\Gamma}^{\star}| \leq Jk+h\big| |\tilde{V}|-|V| \big|.
\end{equation}
\end{lemma}

\begin{proof}
W.l.o.g.\ assume that $|\tilde{V}|\geq|V\|$. Given the labelling of the vertices
that satisfies the above condition, let $\gamma\colon\,\boxminus\to\boxplus$,
denoted by $\gamma=\left(\gamma_{1},\ldots,\gamma_{m}\right)$, be an
optimal path for the Glauber dynamics on $G$. Now let $\tilde{\gamma}\colon\,
\tilde{\boxminus}\to\tilde{\boxplus}$ be the Glauber path of configurations on
$\tilde{G}$, denoted by $\tilde{\gamma}=(\tilde{\gamma}_{1},\ldots,\tilde{\gamma}_{m},
\tilde{\gamma}_{m+1},\ldots,\tilde{\gamma}_{m+|\tilde{V}|-|V|})$, and defined by
the following rule: whichever vertex $v\in V$ is flipped at step $i$ in the path
$\gamma$, flip the corresponding vertex $\tilde{v}\in\tilde{V}$ also at step $i$
in $\tilde{\gamma}$. For steps $m+1,\ldots,m+|\tilde{V}|-|V|$, flip the remaining
$-1$ valued vertices in any arbitrary order. Then it follows that, for $1\leq i\leq m$,
\begin{equation}
\tilde{\mathcal{H}}\left(\tilde{\gamma}_{i}\right)
-\tilde{\mathcal{H}}\left(\tilde{\boxminus}\right)
=J|\tilde{E}\left(\tilde{\gamma}_{i},\overline{\tilde{\gamma}_{i}}\right)|
-h\left|\tilde{\gamma}_{i}\right|
\leq J\left(|E\left(\gamma_{i},\overline{\gamma_{i}}\right)|+k\right)
-h\left|\tilde{\gamma}_{i}\right|.
\end{equation}
Similarly, for $m\leq i\leq m+|\tilde{V}|-|V|$, we have
\begin{equation}
\tilde{\mathcal{H}}\left(\tilde{\gamma}_{i}\right)
-\tilde{\mathcal{H}}\left(\tilde{\boxminus}\right)
\leq Jk-h\left(\left|\gamma_{i}\right|-\left|\boxplus\right|\right).
\end{equation}
It follows that $\tilde{\Gamma}^{\star}\leq\Gamma^{\star}+Jk-h(|\tilde{V}|-|V|)$.
A similar argument gives $\Gamma^{\star}\leq\tilde{\Gamma}^{\star}+Jk
+h||\tilde{V}|-|V||$.
\end{proof}

Now let $G=\left(V,E\right)$ and $\tilde{G}=(\tilde{V},\tilde{E})$ be two
Configuration Models and suppose w.l.o.g.\ that the total degree of the
vertices in $V$ is $\ell_{V}$ and the total degree of vertices in $\tilde{V}$
is $\ell_{\tilde{V}}\geq\ell_{V}$. Let $G_{t}$ and $\tilde{G}_{t}$ be the
extension of each these two graphs, obtained by adding vertices
$\left\{ v_{1},\ldots,v_{t}\right\} $ and $\left\{ \tilde{v}_{1},\ldots,\tilde{v}_{t}\right\}$,
both with the same degree sequence $\left\{d_{1},\ldots,d_{t}\right\}$.
We will couple the construction leading to the two graphs $G_{t}$ and
$\tilde{G}_{t}$ in the following manner: for $1\leq i\leq\sum_{k=1}^{t}d_{k}$,
choose $u_{i}$ uniformly as described above, and pair $i$ with $u_{i}$
in $G_{t}$. Let $\delta_{i}=\left(\ell_{\tilde{V}}-\ell_{V}\right)/\left[\left(\ell_{\tilde{V}}
+i-1\right)\left(\ell_{V}+i-1\right)\right]$ and set $\tilde{u}_{i}=u_{i}$ with probability
$1-\delta_{i}$, and with probability $\delta_{i}$ independently and uniformly
pick one of the remaining $\left(\ell_{\tilde{V}}-\ell_{V}\right)$ points. Then
\begin{equation}
\mathbb{E}[|E_{t}\nabla\tilde{E_{t}}|]
\leq |E\nabla\tilde{E}|+\sum_{i}\delta_{i} \leq |E\nabla\tilde{E}|
+2\left(\ell_{\tilde{V}}-\ell_{V}\right).
\end{equation}
Hence, from Markov's inequality and from Lemma~\ref{lem: gammadifference},
it follows that, w.h.p. and for any function $f(t)$ such that $\lim_{t\to\infty}
f(t)=\infty$,
\begin{equation}
\left|\tilde{\Gamma^{\star}}_{t}-\Gamma^{\star}_{t}\right|
\leq J\left(|E\nabla\tilde{E}|+2(\ell_{\tilde{V}}-\ell_{V})+f(t)\right)
-h\left(\big||\tilde{V}|-|V|\big|\right).
\end{equation}
Hence, by this pairing scheme, we have that $\tilde{\Gamma^{\star}_{t}}/\Gamma^{\star}_{t}
\to 1$ as $t\to \infty$.


\section*{Appendix A}

Note that any configuration $\sigma$ chosen uniformly from all configurations of
size $|\sigma|$ satisfies ($\mathbb{E}$ denotes expectation w.r.t.\ bond percolation)
\begin{equation}
\mathbb{E}\left(\left|E\left(\sigma,\overline{\sigma}\right)\right|\right)
= \mu_{\left|\sigma\right|} \quad \mbox{ with } \quad
\mu_{\left|\sigma\right|} = p\left|\sigma\right|\left(n-\left|\sigma\right|\right).
\end{equation}
Using Chernoff's inequality and a union bound, we can show that if $|\sigma|=\Theta(n)$,
then
\begin{equation}
\mathcal{H}\left(\sigma\right)-\mathcal{H}\left(\boxminus\right)
= J\mu_{\left|\sigma\right|}[1\pm o(1)]-h\left|\sigma\right|.
\end{equation}
Furthermore, any $\sigma$ of size $|\sigma|\leq[1-o(1)]\frac{n}{2}$ has (modulo small
fluctuations) a downhill path to $\boxminus$ (e.g. by flipping $+1$ spins in any arbitrary
order), while every $|\sigma|\geq[1+o(1)]\frac{n}{2}$ has a downhill path to $\boxplus$
(e.g.\ by flipping $-1$ spins in any arbitrary order). This proves hypothesis (H), and the
claim in \eqref{eq:er1}.



\end{document}